\documentclass[%
 aip,
 amsmath,amssymb,
 reprint,%
]{revtex4-1}

\usepackage{graphicx}
\usepackage{dcolumn}
\usepackage{bm}

\usepackage[utf8]{inputenc}
\usepackage[T1]{fontenc}
\usepackage{mathptmx}
\usepackage{etoolbox}

\usepackage{xcolor}
\usepackage[normalem]{ulem}

\usepackage{amsmath}
\usepackage{amsthm}

\def\R{{\mathbb R}}

\def\Z{{\mathbb Z}}

\def\N{{\mathbb N}}

\def\fR{{\mathfrak R}}

\def\mt{{\mathfrak T}}

\def\fy{{\mathfrak Y}}
\def\fz{{\mathfrak Z}}

\def\cR{{\mathcal R}}
\def\cI{{\mathcal I}}

\def\cK{{\mathcal K}}

\def\cg{{\mathcal G}}

\def\ct{{\mathcal T}}

\def\cz{{\mathcal Z}}

\def\ce{c_\ve}

\def\cQ{{\mathcal Q}}
\def\tcq{\tilde\cQ}
\def\cX{{\mathcal X}}

\def\fe{{\mathfrak E}}
\def\fen{{\mathfrak E}^{(n)}}
\def\fent{{\hat{\mathfrak E}}^{(n)}}
\def\G{{\mathcal G}}
\def\T{{\mathcal J}}

\def\ss{{\mathcal S}}

\def\mr{\mathfrak R}

\def\ve{\varepsilon}

\def\a{\alpha}

\def\b{\beta}

\def\ga{\gamma}

\def\l{\lambda}

\def\Ga{\Gamma}

\def\La{\Lambda}

\def\P{{\mathcal P}}

\def\pr{{\mathbb P}}

\def\E{{\mathcal E}}
\def\Er{{\mathbb E}}

\def\nn{\nonumber}

\def\d{\delta}
\def\b{\beta}

\def\ga{\gamma}

\def\={&=&}

\def\vtp{\vec\ct_k^{\,\prime}}

\def\vtpp{\vec\ct_k^{\,\prime\prime}}

\newtheorem{teorema}{Theorem}[section]
\newtheorem{lema}[teorema]{Lemma}

\newtheorem{observacao}[teorema]{Remark}

\newtheorem{teo}{Theorem}[section]
\newtheorem{coro}{Corollary}[section]

\newtheorem{defin}{Definition}[section]

\makeatletter
\def\@email#1#2{%
 \endgroup
 \patchcmd{\titleblock@produce}
  {\frontmatter@RRAPformat}
  {\frontmatter@RRAPformat{\produce@RRAP{*#1\href{mailto:#2}{#2}}}\frontmatter@RRAPformat}
  {}{}
}%
\makeatother
\begin{document}

\preprint{AIP/123-QED}

\title[Scaling limits of the Bouchaud and Dean trap model on Parisi's tree in ergodic and aging time scales]{Scaling limits of the Bouchaud and Dean trap model on Parisi's tree in ergodic and aging time scales}
\author{Luiz Renato Fontes }
\email{lrfontes@usp.br}
\thanks{Partially
		supported by CNPq grant 307884/2019-8, and FAPESP grants 2017/10555-0 and 2023/13453-5}
\affiliation{IME-USP, Rua do Mat\~ao 1010, 05508-090
		S\~ao Paulo SP,  Brazil.}
        
\author{Andrea Hernández}%
 \thanks{Supported by	a CAPES institutional fellowship}
\affiliation{ 
IME-USP, Rua do Mat\~ao 1010, 05508-090
		S\~ao Paulo SP,  Brazil.
}%


\date{\today}

\begin{abstract}
	We take scaling limits of the Bouchaud and Dean trap model on Parisi's tree in time scales where the dynamics is 
	either {\em ergodic} (close to equilibrium) or {\em aging} (far from equilibrium). These results follow from a
	continuity theorem  formulated for a certain kind of process on trees, which we call a 
	{\em cascading jump evolution}, defined in terms of a collection of {\em jump functions}, with a cascading structure 
	given by the tree.
\end{abstract}

\maketitle

%

\section{Introduction}
\label{intro}

As a means to obtain an understanding of dynamical effects in disordered systems such as spin glasses at low temperature, Bouchaud and Dean  \cite{BD} proposed a phenomenological model for the dynamics of such a system, framed within a tree structure inspired in Parisi's Replica Symmetry Breaking (equilibrium) picture \cite{MPV} as follows. 

Consider a tree with $k$ generations/levels/layers rooted at $\emptyset$ in generation 0. Each site of generation $i$, $0\leq i<k$, has $M_{i+1}$ descendants in generation $i+1$,
where $M_j\geq1$, $1\leq j\leq k$. Let us denote a site of the $j$-th generation, $1\leq j\leq k$, by $\bar x_j=(x_1,\ldots, x_j)$, with $1\leq x_j\leq M_j$. We equip each site $\bar x_j$, $1\leq x_j\leq M_j$,  with a time line where a Poisson point process of rate $\l^j_{\bar x_j}$ is bestowed. Poisson processes at different time lines are independent.
The family of rates $\{\l^j_{\bar x_j};\,1\leq x_j\leq M_j,\,1\leq j\leq k\}$ is itself constituted of independent random variables such that $\l^j_{\bar x_j};\,1\leq x_j\leq M_j$ are identically distributed, $1\leq j\leq k$, and the distribution of $\tau^j_{\bar x_j}=(\l^j_{\bar x_j})^{-1}$ is attracted to an $\a_j$-stable law, with $\a_j\in(0,1)$, $1\leq j\leq k$.

The dynamics takes place on the leaves of the tree (namely on the sites $\bar x_k$, $1\leq x_j\leq M_j$, $1\leq j\leq k$), and it is a Markov jump process, as follows. When at $\bar x_k$ at time $t\geq0$, the process waits for the first Poisson mark after $t$ among the time lines of $\bar x_j$, $1\leq j\leq k$ (that is, it looks for the next Poisson mark among $\bar x_k$ and its ascendants up to the first generation). Let $T$ be the time corresponding to such a mark and $J$ be the generation of the ascendant to whose time line the corresponding mark belongs. The jump is then performed at time $T$  from $\bar x_k$ to a descendant of  $\bar x_{J-1}$ on the leaves of the tree chosen uniformly at random. Note that the destination of the jump is of the form $(x_1,\ldots, x_{J-1}, X_J,\ldots, X_k)$, with $X_j$ independent and  uniform in $\{1,\ldots,M_j\}$, 
$1\leq j\leq k$; when $J=1$, we put $\bar x_{J-1}=\emptyset$, and omit it. (The randomness involved in each jump is independent from jump to jump.)
We will refer to this model as the Bouchaud and Dean Trap Model (on Parisi's Tree) ---BDTM for short, or $k$-BDTM when we want to make the number of levels explicit.

In this paper we will take the scaling limit of this dynamics as time and {\em volumes} $M_1,\ldots,M_k$ diverge suitably, under three different regimes, one which we will call {\em ergodic}, and two other, {\em aging}, regimes. 

\subsection{Previous results}
\label{previous}

One of the dynamical effects targeted in \cite{BD} is the {\em aging} phenomenon, and the authors effectively show, in Section B of \cite{BD}, that under appropriate scaling, a certain two-time correlation function of the dynamics, namely the probability that there is no jump of the dynamics during a given time interval $(t_w,t_w+t)$, converges as volumes and time diverges in a certain way to a(n explicit) function of the quotient $t/t_w$, which signals (normal) aging. The authors explicitly require that $\a_1<\cdots<\a_k$ for their result, even though this condition does not seem to be necessary in their computation (nor indeed in our present analysis) --- see Remark~\ref{rmk:order}. 

At this point it should perhaps be said that the reasoning for those results, even while convincing, and yielding the correct results, seem to us to be not fully rigorous, at least not in the usual sense adopted in the mathematics literature. Addressing the latter issue is not however the main motivation of our study. We are indeed more interested in understanding the asymptotic behavior of the {\em full} dynamics (not only one of its specific correlation functions) in different scaling regimes, and in particular how the aging phenomenon can be obtained as a property of a limiting such dynamics.

This program has been carried out before for a variety of related models, like the $k=1$ case of the present dynamics and related 1-level dynamics \cite{FM08, FL1, BFGGM, FP13} ---see also \cite{ABC,ABG1,ABG2, BF} for related mathematical results for the same or related 1-level dynamics. A model which extends the 1-BDTM to multiple levels in a different way was proposed by Sasaki and Nemoto~\cite{SN}, who also derived aging results similar to those in~\cite{BD} with a similar approach ---more about it and its relation to the BDTM in Subsection~\ref{ssec:sn} below---; \cite{FGG} has the analysis of the full dynamics in a single regime, the ergodic one; \cite{GG15} looks at complementary regimes (in a sense); \cite{FG,FFZ} study the 2-level such dynamics in the hypercube, and with a set of random parameters related to the GREM \cite{D80,D85} in different regimes. The GREM-like Trap Model with infinitely many levels (in the limit) was considered in~\cite{FP21,GG15} ---it is a tricky issue to even make sense of the limiting dynamics for this model; we will see below that this is trivial for the BDTM; see Remark~\ref{rmk:inf_levels}.

A farther departure is the $p$-spin dynamics, whose aging behavior is analyzed in \cite{BG}. We could also mention models in $\Z^d$ (rather than the tree, the hypercube, or related graphs), where aging results have been established in varied situations \cite{FIN, AC05,  ACM, AC07, BC, FM14}.

\subsection{Relation to the Sasaki and Nemoto/GREM-like Trap Model}
\label{ssec:sn}

The 1-BDTM is related to REM \cite{D80} in the sense that its equilibrium distribution has the same infinite volume form as the infinite volume Gibbs measure for the REM. 
It is not clear to us whether Bouchaud and Dean had any intention to have their multilevel dynamics correspond to a particular spin glass model in this way. In any case, we may consider the equilibrium distribution of the $k$-BDTM for $k\geq2$ and attempt to take its infinite volume limit. That distribution is given as follows. Let us consider the $k=2$ case for simplicity. The weight of $\bar x_2$ is given by\footnote{The derivation of~\eqref{eqbd} is fairly straightforward, and will not play a role in our results, so we leave it as an exercise to the interested reader.}
\begin{equation}\label{eqbd}
	\frac{\tau^1_{x_1}}{\sum_{y_1=1}^{M_1}\tau^1_{y_1}}\,
	\frac{\frac{\tau^2_{\bar x_2}}{\tau^1_{x_1}+\tau^2_{\bar x_2}}}
	{\sum_{y_2=1}^{M_2}\frac{\tau^2_{x_1y_2}}{\tau^1_{x_1}+\tau^2_{x_1y_2}}}.
\end{equation}
In order to take the infinite volume limit of~\eqref{eqbd}, we need to specify how the volumes diverge in relation to one another; we will adopt a {\em fine tuning} such relation (here and also many times throughout) such that the maxima of $\tau^1_\cdot$ and $\tau^2_\cdot$ variables have the same order of magnitude. The resulting limit is thus a (random) probability distribution whose weights are given by
\begin{equation}\label{eqbd_inf}
	\frac{\ga^1_{x_1}}{\sum_{y_1\in[0,1]}\ga^1_{y_1}}\,
	\frac{\frac{\ga^2_{\bar x_2}}{\ga^1_{x_1}+\ga^2_{\bar x_2}}}
	{\sum_{y_2\in[0,1]}\frac{\ga^2_{x_1y_2}}{\ga^1_{x_1}+\ga^2_{x_1y_2}}},
\end{equation}
where $\ga^1_{x_1}$, $x_1\in[0,1]$, are the increments of an $\a_1$-stable subordinator; and for each $x_1\in[0,1]$, $\ga^2_{\bar x_2}$, $x_2\in[0,1]$, are the increments of an $\a_2$-stable subordinator; all the subordinators are collectively independent. It becomes clear that this differs from the corresponding infinite volume GREM Gibbs measure (in the {\em cascading phase}, with the appropriate choice of parameters~\cite{FG}), given by
\begin{equation}\label{eqsn_inf}
	\frac{\ga^1_{x_1}\ga^2_{\bar x_2}}{\sum_{\bar y_2\in[0,1]^2}\ga^1_{y_1}\ga^2_{\bar y_2}}.
\end{equation}

It is remarkable that in order that~\eqref{eqsn_inf} makes sense as a probability, we must have $\a_1<\a_2$, otherwise the denominator diverges almost surely, but this is {\em not} the case for the denominators in~\eqref{eqbd_inf}, which converge almost surely.

This point notwithstanding, Sasaki and Nemoto, arguing that the BDTM "is not so realistic", since it allows for the jump  from a given $\bar x_k$ involving an older ascendant of $\bar x_k$ (closer to $\emptyset$) to be more probable than the jump involving a younger ascendant, proposed another multilevel model~\cite{SN}, which turns out to have its infinite volume equilibrium measure correspond to the GREM infinite volume Gibbs measure~\cite{Gav}, and for this reason was called (in~\cite{FGG}) the GREM-like Trap Model (GLTM). We refer to ~\cite{SN, FGG} for detailed descriptions of the model (as well as the results to which we have already alluded to above ---see also~\cite{FG}). 

The point raised by Sasaki and Nemoto indeed holds in finite volume, but as we will see below, {\em it does not} for the scaling limits we derive here, as discussed below ---see Remark~\ref{rmk:sn}. Be that as it may, it seems like a worthy project, undertaken in this paper, to strive to understand the scaling, in particular aging, behavior of the BDTM from a fully dynamical point of view, similarly to what has been done for the GLTM and related models, in order to close a gap, or at least shed light on a model formulated in a foundational paper for  many of the mathematics and physics literature of disordered systems related to aging.

\subsection{Model and results}
\label{mod}

We have already described the model in detail at the introduction, but it turns out to be convenient  to represent the dynamics not as we did, as the current leaf $\bar x_k$. Indeed it is natural in the context of aging to represent it as $\tau^k_{\bar x_k}$, as has been occasionally done in the literature, but as will become clear below, it is more convenient and also natural to represent it as the vector
$(\Ga^1_{x_1},\ldots,\Ga^k_{\bar x_k})$, where for $1\leq j\leq k$, $\Ga^j_{\bar x_j}=(\La^j_{\bar x_j})^{-1}$, and $\La^j_{\bar x_j}=\sum_{i=1}^j\l^i_{\bar x_i}$. It follows from elementary and well known properties of the exponential distribution that, once the BDTM is at $\bar x_k$ at a given time, $\Ga^j_{\bar x_j}$ represents the expected (remaining) waiting time thence till the first jump `beyond level $j+1$', that is, a jump for which $J\leq j$ in the description at the introduction. 
This is relevant for aging, particularly when computing `no jump' correlation functions, usually considered in this context.

It is also convenient, in order to take scaling limits of the model, to represent it in terms of {\em clock processes}, which are the most natural objects to take scaling limits of, from a technical point of view. Ordinary clock processes record cumulative successive {\em sojourns} of the dynamics (i.e., durations of visits to the successive states visited by the process, along its trajectory). These are most conveniently represented as a `jump function', i.e., a function which increases only by jumps, and whose jump sizes correspond to the sojourns of the process. These processes are the central objects of many aging results in the mathematical literature, since in particular  `no jump' correlations can be expressed in terms of them alone.
But our program requires more than only the jump times; we want to record also the state of the process during each sojourn. We do that in terms of another jump function, which we call (in lack of a better term) the `timeless' clock process. Let us see how it works for $k=1$. A full description will be done later on.

Let $N^1_{x_1}$, $1\leq x_1\leq M_1$, be iid Poisson rate 1 counting processes, and make
\begin{equation}\label{clocks}
	S_1(r)=\sum_{x_1=1}^{M_1}\tau^1_{x_1}N^1_{x_1}(r),\qquad C_1(r)=\sum_{x_1=1}^{M_1}\tau^1_{x_1}\sum_{j=1}^{N^1_{x_1}(r)}E_{x_1}^j,
\end{equation}
where $\{E_{x_1}^j;\,x_1,j\geq1\}$ are iid standard exponential random variables (and $\sum_{j=1} ^0\cdots=0$, by convention).
$C_1$ is the (ordinary) clock process, and $S_1$ is the timeless clock process. Notice that they are both jump functions with the same jump points. 
{The 1-level version of the BDTM is then defined as}
\begin{equation}\label{kp}
	Z_1(t)=S_1(C^{-1}_1(t))-S_1(C^{-1}_1(t)-), t\geq0,
\end{equation}
where, here and below, $C_\cdot^{-1}$ is the (right continuous) inverse of $C_\cdot$. 
We may also, when it is convenient,
look at 
{$Z_1$} 
within a finite time interval $[0,L]$, with $L=\La^{-1} E$, with $E$ a standard exponential random variable independent of all else, and $\La>0$ arbitrary, so that in~\eqref{kp} we may have $0\leq t\leq L$.
See Figure~\ref{fig:cclocks}.

\begin{figure*}[!htb]
	\centering
	\includegraphics[scale=0.5,trim=30 55 0 0,clip]{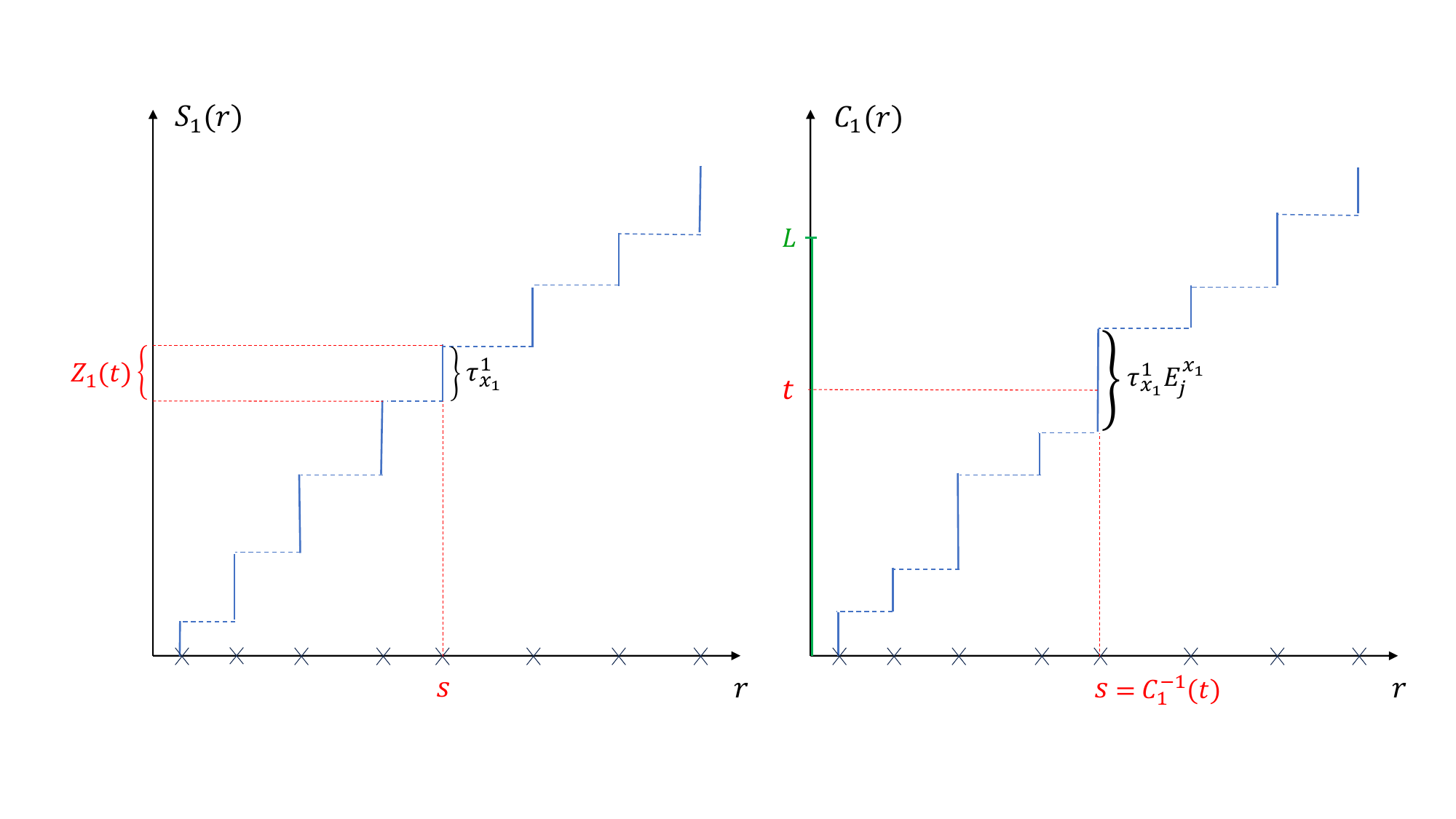}
	\caption{Descriptive illustration of $Z_1(t)$ for a given $t$, based on the timeless clock process $S_1$ and the ordinary clock process $C_1$. The crosses in the $x$-axis represent the marks of the Poisson processes $N^1_{x_1}$, $1\leq x_1\leq M_1$, and the locations of the  jumps of both processes; jump sizes are represented in full lines; in this picture, $s$ corresponds to the $j$-th mark of $N^1_{x_1}$, and we may say that $s$ is labeled $x_1$. Notice that the labels of the successive jumps are iid random variables, uniform in $\{1,\ldots, M_1\}$.}
	\label{fig:cclocks} 
\end{figure*}
Let us note that the (real) time of the process runs on the $y$-axis of the graph of $C_1$ in Figure~\ref{fig:cclocks}; the $x$-axis $r$ argument is merely auxiliary. 

We also note that we may represent $C_1$ in terms of $S_1$ and an iid family of standard exponential random variables as follows
\begin{equation}\label{clock1}
	C_1(r)=\int_0^rE_s\,dS_1(s),
\end{equation}
where $\{E_s,\,s\geq0\}$ is an iid family of standard exponential random variables, and the integral is the Lebesgue-Stieltjes integral with respect to the $S_1$ viewed as a distribution function.

We may then seek to obtain scaling limits for $Z_1$ via suitable scaling limits of $S_1$, as follows. Let $S_1^{(n)}$ be a suitable rescaling  of $S_1$, and suppose that $S_1^{(n)}$ converges suitably to a jump function $S'_1$ as $n\to\infty$. Then, as soon as we have a suitable {\em continuity theorem} at our disposal, we find that $C_1^{(n)}$ converges in distribution to $C'_1(\cdot)=\int_0^\cdot E_s\,dS'_1(s),$ as $n\to\infty$, and the correspondingly rescaled $Z_1^{(n)}$ converges in distribution as $n\to\infty$ to $Z'_1$ obtained from $S'_1$ and $C'_1$ as in~\eqref{kp}.

This program (for $k=1$) was indeed carried out along these lines in~\cite{BFGGM}, and we will adopt a similar strategy to treat the the multilevel case, by first proving a continuity theorem (as suggested above) in a more abstract setting, which we then apply to find scaling limits for the dynamics in a few regimes (as anticipated above).

Before going to the details of this approach, in the next section, let us 
provide the multilevel description of the BDTM, extending the above $k=1$ case, thus allowing for the subsequent statement of out main results by the end of the present section.

The simplicity of the BDTM (as compared to the GLTM) is that the its dynamics is {\em nested}, in the sense that we may build the process level by level, from the first one up, with the construction on a given level not affecting what has been done in lower levels. 
Indeed, if in the $k$-level dynamics, as defined at the beginning of the introduction, we ignore levels $j+1$ to $k$, for some $j<k$, and focus on the levels 1 to $j$ only, we see the $j$-level dynamics. In particular, for $j=1$, we have the 1-level dynamics as described in this subsection.
Once we have defined the level-1 dynamics, as above, let us define the level-2 dynamics (ignoring the upper levels, which, again, we may) independently within each sojourn interval of the level-1 dynamics, and similarly  to the  level-1 dynamics, as follows.

Let us look at a particular sojourn, say the one corresponding to jump $s$ in Figure~\ref{fig:cclocks}. The level-2 dynamics during that time interval follows the same pattern as that of the level-1 dynamics during the time interval of length $L$. Notice that  the sojourn interval in question has length of the same form of $L$ (i.e., it is exponentially distributed). So resetting the origin of that time interval from $C_1(s-)$ to $0$ and making $S_2(r|s)=\sum_{x_2=1}^{M_2}\tau^2_{\bar x_2}N^2_{s,x_2}(r)$, with $\{N^2_{s,x_2},\, 1\leq x_2\leq M_2\}$ iid rate 1 Poisson counting processes, independent of all else, we may define $Z'_2(t|s)=S_2(C^{-1}_2(t|s)|s)-S_2(C^{-1}_2(t|s)-|s), 0\leq t\leq\tau_1^{x_1}E_j^{x_1}$. This gives a dynamics $Z'_2(\cdot|s)$ inside the sojourn interval corresponding to $s$ (with the origin translated back to the origin). Notice it corresponds to the 2-level dynamics inside that sojourn interval as described at the introduction. In order to get a full dynamics $Z_2'$, as it should be quite clear, it is only a matter of concatenating the dynamics within the sojourn intervals in their natural order, that is, $Z'_2(t)=Z'_2(t-C_1(s-)|s),\,C_1(s-)\leq t<C_1(s)$, with s varying over the jump locations of $C_1$. 
Notice that the successive labels on the first level are uniformly distributed in $\{1,\ldots, M_1\}$ (as pointed out in Figure~\ref{fig:cclocks}), and within first level sojourn intervals, the second level labels are uniformly distributed in $\{1,\ldots, M_2\}$,  consistently with  the description of the dynamics made at the introduction.
In order to get $Z_2$, it is just a matter of mapping the $\tau_\cdot^\cdot$ values into the partial sums, according to the indication at the first paragraph of this subsection, namely $Z_2(t)=((Z_1(t))^{-1}+(Z'_2(t))^{-1})^{-1}$. That this is a natural representation of the 2-level BDTM will be argued below, but it is also a sensible representation from a technical point, as will be argued later on.

A similar procedure defines the process on subsequent levels, by constructing the level-$j+1$ dynamics within sojourn intervals of level-$j$, using the (random) parameters of level $j+1$, 
{namely, $\tau^{j+1}_{\cdot x_{j+1}}$}, independently of other such intervals, and of objects in previous levels, and then concatenating.  We end up with $\cz(t)=\cz_k(t):=(Z_1(t),\ldots,Z_k(t))$, $0\leq t\leq L$
{(or $t\geq0$), which we call the $k$-BDTM with parameter set $\{\tau^i_{\bar x_i}, 1\leq x_i\leq M_i,\,1\leq i\leq k\}$}. 

A more detailed construction follows from the discussion in the next section, where a  more general process is defined and studied. Let us next present our scaling limit results. In order to do that, we must first define the limiting dynamics which come up in the different regimes we consider.

\subsubsection{Limiting dynamics}
\label{limits}

We have two limiting dynamics for the scaling regimes we will consider. Both are defined from particular (random) jump functions in much the same way as we defined the BDTM, level by level, within a nested, cascading structure.

\paragraph*{\textbf{Cascading K process.}}

Let $\ga^1_{x_1}$, $x_1\in[0,1]$, be the pointwise increments of an $\a_1$-stable subordinator 
{\footnote{
{We recall that a subordinator is a non decreasing process whose increments are independent and stationary; an $\a$-stable subordinator have increments with an $\a$-stable distribution\cite{B99}.}
}}
%
in $[0,1]$, and let $N^1_{x_1}$,  $x_1\in[0,1]$, be iid rate 1 Poisson counting processes. 
Set
\begin{equation}\label{clocks-k}
	S_1(r)=\sum_{x_1\in[0,1]}\ga^1_{x_1}N^1_{x_1}(r),\qquad C_1(r)=\sum_{x_1\in[0,1]}\ga^1_{x_1}\sum_{j=1}^{N^1_{x_1}(r)}E^{x_1}_j,
\end{equation}
similarly as in~\eqref{clocks}, and define $Z_1$ as in~\eqref{kp}. This defines the level-1 dynamics, and 
{the behavior of the dynamics} on subsequent levels follows the nested pattern of the BDTM: it is defined within each sojourn interval of the previous level, say $j-1$, as a level-1 type dynamics, independent of all other intervals and levels, using an $\a_j$-stable subordinator  
{associated to the label of the sojourn interval}, 
instead, 
{with independence among subordinators of different levels and labels}. We call the resulting $\cz(t)=(Z_1(t),\ldots,Z_k(t))$ process as the {\em cascading K-process} with  (multi)index $\bar \a_k=(\a_1,\ldots,\a_k)$. 


\paragraph*{\textbf{Cascading aging process.}}

Let $\ga^1_{x_1}$, $x_1\in[0,\infty)$ be the increments of an $\a_1$-stable subordinator in $[0,\infty)$. Set
\begin{equation}\label{clocks-a}
	S_1(r)=\sum_{x_1\in[0,r]}\ga^1_{x_1},\qquad C_1(r)=\sum_{x_1\in[0,r]}\ga^1_{x_1}E_{x_1},
\end{equation}
similarly as in~\eqref{clocks}, and define $Z_1$ as in~\eqref{kp}. This defines the level-1 dynamics, and 
{the behavior the dynamics on subsequent levels follows the nested pattern of the BDTM, and the K process, as described above.} 
We call the resulting $\cz(t)=(Z_1(t),\ldots,Z_k(t))$ process as the {\em cascading aging process} with index $\bar \a_k$.

{
\begin{observacao}\label{rmk:ka}
	The K and the aging processes have an obvious similar cascading structure; an important difference is that while the K process visits the same states repeatedly (in each level), there is never any return to the same state in the aging process. Other distinguishing properties of these processes are discussed below.
\end{observacao}
}

\subsubsection{Scaling regimes and main results}
\label{scaling-results}

Recall that we are assuming that 
\begin{equation}\label{tau}
	\pr(\tau^j_{\bar x_j}>t)=\frac{L_j(t)}{t^{\a_j}},\,t\geq0
\end{equation}
where $L_1,\ldots,L_k$ are slowly varying at infinity, and $\a_j\in(0,1)$, $1\leq j\leq k$. For simplicity, we will assume that $\lim_{t\to\infty}L_j(t)$ exists and is a constant, which we take to be 1, again for simplicity. See a similar treatment of the general case in~\cite{H}.

In order to define our scaling regimes, we need to explain how the time and volumes diverge in relation one to the other. In two of the three regimes we consider, we adopt a {\em fine tuning} relation of the volumes among themselves, in such a way as to have the $\tau$ variables of different levels {\em scale together}  or {\em simultaneously}
{in the same way}. We make 
{\begin{equation}\label{M}
M_1:=n \,\text{ and }\, M_j=n^{\a_j/\a_1}, 2\leq j\leq k
\end{equation}
(so that $M_j^{1/\a_j}$ is constant in $j$).
}
With this choice of volumes, we have two time scalings (with respect to the the scaling parameter $n$), defining two scaling regimes, as follows.
{\footnote{
{See Remark~\ref{nft} at the end for a brief discussion of other volume scalings.}}}

\paragraph*{$\bullet$ \textbf{Ergodic time scaling.}}

Let 
{$\cz^{(n)}$ be the $k$-BDTM with parameter set $\{\tau^i_{\bar x_i}, 1\leq x_i\leq M_i,\,1\leq i\leq k\}$ as in~(\ref{tau},\ref{M}) above}. Set $\tilde\cz^{(n)}(t)=(\tilde Z_1^{(n)}(t),\ldots,\tilde Z_k^{(n)}(t))$, where $\tilde Z_j^{(n)}(t)=\frac1{n^{1/\a_1}}Z_j(n^{1/\a_1}\,t)$, $t\geq0$, $1\leq j\leq k$, $n\geq1$.

\begin{teo}\label{teo:erg}
	Under the conditions stipulated above, namely~\eqref{tau} (with the convergent $L_j$'s) and fine tuning, we have that $\tilde\cz^{(n)}(\cdot)$ converges in distribution as $n\to\infty$ to $\tilde\cz$, the cascading K-process with index $\bar \a_k$.
\end{teo}

The distribution here is the joint distribution of the environment random variables $\tau_\cdot^\cdot$ and subordinators, and the dynamical random variables, namely the Poisson processes and standard exponentials. The topology (here and elsewhere) is the $J_1$ Skorohod topology.

The terminology for this scaling regime comes from the fact that the cascading K-process is an ergodic Markov process (with~\eqref{eqbd_inf}, or an extension of it to more levels, as the equilibrium measure. We do not go into these details, but the issue can be addressed similarly as in~\cite{Gav}\footnote{ where, as already mentioned above, the GLTM was analyzed (in the specific ergodic regime, with fine tuning); \eqref{eqsn_inf}, or an extension of it to more levels, is the equilibrium of the limiting rescaled dynamics}.

\paragraph*{$\bullet$ \textbf{Polynomial aging time scaling.}}

Let 
{$\cz^{(n)}$ be the $k$-BDTM as for Theorem~\ref{teo:erg}}, and let  $\b\in(0,1/\a_1)$ be a fixed parameter. Set $\hat\cz^{(n)}(t)=(\hat Z_1^{(n)}(t),\ldots,\hat Z_k^{(n)}(t))$, where $\hat Z_j^{(n)}(t)=\frac1{n^\b} Z_j(n^\b\,t)$, $t\geq0$, $1\leq j\leq k$, $n\geq1$.

\begin{teo}\label{teo:pserg}
	Under the conditions of Theorem~\ref{teo:erg}, we have that for almost every realization of the environmental random variables $\tau_\cdot^\cdot$, $\hat\cz^{(n)}(\cdot)$ converges in distribution  as $n\to\infty$ to $\hat\cz$, the cascading aging process with index $\bar \a_k$.
\end{teo}

The distribution here is the conditional distribution of $\cz^{(n)}$ given $\tau_\cdot^\cdot$.

\medskip

Our last scaling regime does not require fine tuning.

\paragraph*{$\bullet$ \textbf{Order 1 aging time scaling.}}

Let 
Let $\cz$ be the 
{$k$-BDTM with parameter set $\{\tau^i_{\bar x_i}, 1\leq x_i\leq M_i,\,1\leq i\leq k\}$ as in~(\ref{tau}) above, with $M_1,\ldots,M_k$ not necessarily in fine tuning.} Set $\check\cz^{(m)}(t)=(\check Z_1^{(m)}(t),\ldots,\check Z_k^{(m)}(t))$, where $\check Z_j^{(m)}(t)=\frac1{m}Z_j(m\,t)$, $t\geq0$, $1\leq j\leq k$,  $m\geq1$.

\begin{teo}\label{teo:1serg}
	For almost every realization of the environmental random variables $\tau_\cdot^\cdot$, $\check\cz^{(m)}(\cdot)$ converges in distribution  as first $M_1,\ldots,M_k\to\infty$ (in any way), and then $m\to\infty$,  to $\hat\cz$.
\end{teo}

Notice that the limits in Theorems~\ref{teo:pserg} and~\ref{teo:1serg} are the same.

\begin{observacao}\label{rmk:o1}
	We believe that the limit of Theorem~\ref{teo:1serg} is the one taken in~\cite{BD} (for the aging functions integrated with respect to the distribution of $\tau_\cdot^\cdot$).
	One might say that this is the physical regime, since under experimental conditions volumes should perhaps be considered infinite before taking the time limit. 
	We are not aware of it having been considered before in the mathematics literature for these kinds of models and issues; the
	scaling regimes appearing in Theorems~\ref{teo:erg} and~\ref{teo:pserg} are on the other hand quite common (but see~\cite{AG, BGS}, where shorter time scales are shown to give rise to extremal processes in the limit, rather than subordinators, for $p$-spin dynamics).
\end{observacao}

\begin{observacao}\label{rmk:sn}
	We go back to a comment made at the end of Subsection~\ref{ssec:sn} to point out that in both limiting dynamics $\tilde\cz$ and $\hat\cz$, almost surely, coordinates change at higher levels more easily than at lower levels. We may see changes at higher levels without any change at lower levels, but in order that a lower coordinate of the dynamics changes, the higher ones must change as well.
    {(This is due to the fact that in those processes, almost surely, for every sojourn interval of a given level, we find (infinitely many) jumps of the next level subordinator. In contrast, in finite volume, say the $2$-BDTM, we will find a sojourn interval of the first level with no jump of the associated second level clock process, so that the only jump on this interval occurs at the end, with a jump at the first level.)}
\end{observacao}

\begin{observacao}\label{rmk:inf_levels}
	This is a good point to comment on an issue raised in Subsection~\ref{previous}. There is no particular reason to stop the construction of  $\tilde\cz$ and $\hat\cz$ at a fixed level $k$. We may continue the inductive procedure, level by level, to indefinitely many levels, and thus obtain infinite level dynamics $(\tilde Z_1(t), \tilde Z_2(t),\ldots)$ and $(\hat Z_1(t), \hat Z_2(t),\ldots)$ (provided all the $\a_j$'s are in $(0,1)$).
	
	We note that the mere definition of a meaningful notion of an infinite volume, infinite level dynamics for the GLTM, as undertaken in~\cite{FP21}, required nontrivial, extensive considerations; the same can be said for the $k\to\infty$ limit taken in~\cite{GG15}.
\end{observacao}

\subsubsection{Aging results}
\label{aging results}

The latter two theorems, as well as the scaling properties of stable subordinators, allow for obtaining aging results in the latter two scaling regimes, as follows.

Following the terminology of~\cite{BD}, for $1\leq j\leq k$, and $t_w,t>0$, let 
\begin{equation}\label{pi}
	\Pi_j(t,t_w)=\pr(\text{no jump of }Z_j \text{ during }[t_w,t_w+t]).
\end{equation}

As corollaries of Theorems~\ref{teo:pserg} and \ref{teo:1serg}, under their respective conditions, we have that for almost every realization of  $\tau_\cdot^\cdot$
\begin{widetext}
\begin{align}\label{pip}
	\Pi_j(n^\b t,n^\b t_w)&=\pr(\text{no jump of }\hat Z^{(n)}_j \text{ during }[t_w,t_w+t])\to f_j(t/t_w)\,\text{ as } n\to\infty,\\ 
	\Pi_j(m t,m t_w)&=\pr(\text{no jump of }\check Z^{(m)}_j \text{ during }[t_w,t_w+t])\to  f_j(t/t_w)
	\label{pi1}
\end{align}
\end{widetext}
as first $M_1,\ldots,M_k\to\infty$, and then $m\to\infty$, where 
\begin{equation}\label{pii}
	f_j(\cdot)=\pr(\text{no jump of }\hat Z_j \text{ during }[1,1+\cdot]).
\end{equation}

From the nested structure of $\check \cz$ and scaling properties of subordinators, it follows that $f_j(\cdot)=\bar F_j(\cdot/(1+\cdot))$, where $\bar F_j=1-F_j$,  and $F_j$ is the distribution function of $\prod_{i=1}^jB_i$, with $B_1,\ldots,B_k$ independent random variables, $B_i\sim$ Beta$(1-\a_i,\a_i)$. This agrees with the aging results of~\cite{BD} (in Section B), as far as  conditions are comparable. We discuss this in more detail  in Appendix~\ref{app3}.

\begin{observacao}\label{rmk:j1_cont}
	It is also worth remarking that the $J_1$ topology plays a crucial for establishing the convergences in~\eqref{pip} and~\eqref{pi1} as corollaries of Theorems~\ref{teo:pserg} and~\ref{teo:1serg}, respectively, although this is not immediate, since the 
    {indicator function of the} event in the probability in~\eqref{pii} is almost surely discontinuous for $\hat\cz$ 
    {\footnote{
    {Given a trajectory in such an event, one may approximate it by trajectories in path space which do not are not constant in the given interval, and are thus not in the event.}}}.  We discuss this point in more detail in  Appendix~\ref{app3}, along with other ways of signalling the aging phenomenon in more direct ways from Theorems~\ref{teo:pserg} and~\ref{teo:1serg}.
	%
\end{observacao}

\begin{observacao}\label{rmk:order}
	Going back to a comment at the beginning of Subsection~\ref{previous}, it is perhaps intriguing that~\cite{BD} imposes an ordering condition on the $\a_j$'s ---see Section B---, which is  not absolutely needed for our results, and neither seemingly for the computations of~\cite{BD}; the limiting aging functions are certainly well defined without that.
	Does it originate out of a physicality/naturalness concern, as the one raised in~\cite{SN} (and explained at the end of Subsection~\ref{ssec:sn})? 
\end{observacao}




%

\section{Continuity theorem for cascading jump evolutions}
\label{contcas}

As anticipated above, we will prove Theorems~\ref{teo:erg},~\ref{teo:pserg},~\ref{teo:1serg} via a continuity theorem for an abstract type of process including our process and its rescaled versions. From the discussion above, it is natural to start with jump functions as  the timeless clocks out of which we build ordinary clocks, and then our process, as in~(\ref{clocks}, \ref{kp}), within a nested, cascading structure. 

\subsection{Cascading families of jump functions} \label{cfjf}

Let  $\R_+=[0,\infty)$ and  $D(\R_+)$ be the space of càdlàg real trajectories from $\R_+$ to $\R_+$ equipped with the  $J_1$ Skorohod topology.  Let  $\G:=\{S:[0,\infty)\to[0,\infty)\}$ a collection of functions in $D(\R_+)$ such that $S(0)=0$, $S$ is non-decreasing, $S(s)<\infty$ for all $s\in (0,\infty)$
and  $S(s)\to\infty$ as $s\to\infty$.

For $S$ in $ {\G}$, let
\begin{align}\label{TauS}
	\T_S&:=\{s\geq 0: S(s)-S(s-)=:\gamma(s)>0\}\quad \text{and}\nonumber \\
	\hat{\T_S}&:=\{\gamma(s): s \in \T_S\}
\end{align}
denote the set of jump( location)s and jump sizes of $S$, respectively, and let
\begin{equation}\label{hatG}
	\hat{\G}:=\{S\in\G: S(r)=\sum_{s\leq r}\gamma(s),\quad r\geq 0\}.
\end{equation}
We call a function $S$ in $ \hat{\G}$  a {\it jump function}.

\begin{defin}\label{Sdisordense}
	We say that a jump function  $S$ is {\em dense}  if   $\T_S$ is dense in $[0,\infty)$,  and  {\em discrete} if $\T_S$ is discrete (i.e.,  has no limit points). 
\end{defin}

\begin{defin}\label{defCasFunct}
	Given $k\geq 1$, we define a ($k$-level) {\em cascading family of jump function} as follows. 
	We pick $S_1\in\hat{\G}$, and let $\ct_1 =\T_{S_1}$.  
	For $1\leq j<k$, having defined 
	$\ct_j$, and for each $\bar{s}_j\in\ct_j$, we pick $S_{j+1}(\cdot| \bar{s}_j)\in\hat{\G}$, and let 
	$\ct_{j+1}=\bigcup_{\bar s_j\in\ct_j}\{\bar s_j\}\times\T_{S_{j+1}(\cdot| \bar s_j)}$.

	For $1\leq j\leq k$, we call the 
	family  $\ss_j=\{S_i(\cdot| \bar s_{i-1});\,\bar s_{i-1}\in\ct_{i-1},\,1\leq i\leq j\}$ a ($j$-level) {\em  cascading family of jump  functions} ($j$-CFJF);
	we say that $\ss_j$ is {\em dense} (respectively, {\em discrete}) if all the  functions in it are dense (respectively, discrete); 
	and $\ct_j=\ct_j(\ss_j)$ is called the {\em jump tree} associated to $\ss_j$. 
\end{defin}

Given a $k$-CFJF $\ss_k$, we write for short
\begin{align}
	\T_j(\bar s_{j-1})&=\T_{S_j(\cdot| \bar s_{j-1})},\quad \text{and}\\
	\hat{\T}_j(\bar s_{j-1})&=
	\{\gamma_j(s_j|\bar{s}_{j-1}), s_j\in \T_j(\bar{s}_{j-1})\},
\end{align}
for $\bar s_{j-1}\in\ct_{j-1}$, $j=1,\ldots, k$, where $\ct_0=\{\emptyset\}$,  $S_1(\cdot| \emptyset)=S_1$, $\T_1(\emptyset)=\T_1$, $\hat\T_1(\emptyset)=\hat\T_1$.

We will also use the following notation. For $j=1,\ldots,k$, given  $s_j\in\T_{j}(\bar{s}_{j-1})$, $\bar{s}_{j-1}\in\ct_{j-1}$, 
set
\begin{align}
	\lambda_{j}(s_j|\bar{s}_{j-1})&:= (\gamma_{j}(s_j|\bar{s}_{j-1}))^{-1}\;,\nonumber
    \\
    \Lambda_{j}(s_j|\bar{s}_{j-1})&:= \sum_{i=1}^j\lambda_{i}(s_i|\bar{s}_{i-1})\;, \nonumber
    \\
	 \Ga_{j}(s_j|\bar{s}_{j-1})&:=
      \big(\Lambda_{j}(s_j|\bar{s}_{j-1})\big)^{-1}.\label{deflambda}
\end{align}

\subsection{Cascading jump evolutions} \label{cje}

Given a $k$-CFJF $\ss_k$, 
we define processes $X_1,\ldots,X_k$ inductively as follows.  
Let $\{E_{j}(s|\bar{s}_{j-1}); 1\leq j\leq k, s\in\R_+, \bar s_{j-1}\in\ct_{j-1}\}$ 
be a family of  iid~standard exponential random variables,
and for $1\leq j\leq k$ and $\bar s_{j-1}\in\ct_{j-1}$, define
\begin{equation}\label{K}
	K_{j}(r|\bar s_{j-1})=\sum_{s\in\T_j(\bar s_{j-1})\cap [0,r]}\gamma_{j}(s|\bar s_{j-1})E_j(s|\bar s_{j-1}),\,r\geq0.
\end{equation}
Notice that $K_\cdot(\cdot|\cdot)\in\hat{\G}$ a.s.,  and $\T_{K_{j}(\cdot|\bar s_{j-1})}=\T_j(\bar s_{j-1})$ (but $\hat{\T}_{K_{j}(\cdot|\bar s_{j-1})}\neq \hat{\T}_j(\bar s_{j-1})$).  

Let $C_1=K_1=K_1(\cdot|\emptyset)$, and make
\begin{equation}\label{x1}
	X_{1}(t)=S_1(C_1^{-1}(t))-S_1(C_1^{-1}(t)-),\,t\geq0.
\end{equation}
For $t\geq0$, $X_1(t)$ equals  either a jump size of $S_1$ in $\hat{\T}_1$, corresponding to a jump in $\T_1$ which also is a jump of $C_1$,  or 0.  $X_1$ is a process whose trajectory visits the jumps of $S_1$ "in order", spending an exponential time of mean $\gamma_1(s_1)$ at $s_1\in\T_1 $.

For a given $s_1\in\T_1$, we call the interval $[C_1(s_1-),C_1(s_1))$ the {\em constancy interval of $X_1$ associated to $s_1$},  so  we have that $X_1(t)=\gamma_1(s_1)$ for  $t\in [C_1(s_1-),C_1(s_1))$. We call an interval $I$ a {\em constancy interval of} $X_1$({\em tout court}) if it is a constancy interval of $X_1$ associated to some $s_1\in\T_1$.

\begin{observacao}\label{Lebesgue-Cantor}
	It may be readily checked that  the union of all  constancy intervals of $X_1$ (over all $s_1\in \T_1$) has full Lebesgue measure, but may not equal $\R_+=[0,\infty)$. It may happen,  as in  our applications, later on in this text,  that the complement is a Cantor-like set.
\end{observacao}

Proceeding inductively, and assuming  that $X_i$ is defined for $i=1,\ldots,j-1$,  $2\leq j\leq k$, given a constancy interval $I=[a,b)$ of $X_{j-1}$, for some $a<b$, associated to some 
$\bar{s}_{j-1}\in\ct_{j-1}$, 
for $r\geq 0$, let
\begin{equation}\label{clock}
	C_{j}(r|\bar{s}_{j-1})=a+  K_{j}(r|\bar{s}_{j-1}),             
\end{equation}
and make  
\begin{equation}
	X_{j}(t)=S_j(C_{j}^{-1}(t|\bar{s}_{j-1})|\bar{s}_{j-1})-S_j(C_{j}^{-1}(t|\bar{s}_{j-1})-|\bar{s}_{j-1}),\,t\in I.
\end{equation}

Given a constancy interval of $X_{j-1}$, $I$, associated to $\bar{s}_{j-1}\in\ct_{j-1}$, the constancy interval of $X_j$ associated to 
$\bar{s}_j$ such that ${s}_j\in{\T}_j(\bar{s}_{j-1})$  is the interval 
\begin{equation}\label{constancyintervalj}
	[C_j(s_j-|\bar{s}_{j-1}),C_j(s_j|\bar{s}_{j-1}))\cap I,
\end{equation}
provided the latter intersection is not empty; in the latter interval, $X_j$ equals $\ga_j(s_j|\bar s_{j-1})$.

\begin{observacao}\label{X-process}
	\begin{itemize}
		\item[(1)] The processes $X_1,X_2,\ldots,X_k$ have a nested structure: given $1\leq j<k$ and a constancy interval of $X_j$, $I=[a,b)$, associated to $\bar{s}_j$, the evolution of $X_{j+1}$ in $I$ is determined solely by the {\em partition} of $I$ induced by the range of $C_{j+1}(\cdot|\bar{s}_j)$ in $I$,  namely,  the constancy intervals of $X_{j+1}$ within $I$, and the value of $X_{j+1}$ in each of them are determined solely by $C_{j+1}(\cdot|\bar{s}_j)$ and $S_{j+1}(\cdot|\bar{s}_j)$ (equivalently, by $S_{j+1}(\cdot|\bar{s}_j)$ and the family of iid  standard exponential random variables $E_{j+1}(\cdot|\bar{s}_j)$).  Since distinct constancy intervals of $X_j$ correspond to distinct $\bar{s}_j$'s, and for distinct $\bar{s}_j$'s  the families $E_{j+1}(\cdot|\bar{s}_j)$) are independent, we have that the evolution of $X_{j+1}$ inside distinct constancy intervals of $X_j$ are independent.
		
		\item[(2)] We may then consider the evolution of $X_{j+1}$ inside each constancy interval of $X_j$ separately (and independently),  and  concatenate them,  following the (lexicographic) order of the $\bar{s}_j$'s.
		
		\item[(3)] Given that the intersection in
		(\ref{constancyintervalj})  must be non-empty,  not every $\bar{s}_j\in\ct_j$ corresponds to a constancy interval of $X_j$  if $j\geq 2$. If we consider the evolution of $X_1$ in a finite interval $I$, then, by the same reason, not every $s_1\in\ct_1$ will correspond to a constancy interval of $X_1$ (inside $I$).
		
		\item[(4)] For any $i=1,\ldots,k$ the trajectory of $X_i$ remains unchanged if we rescale the argument of $S_i(\cdot|\cdot)$, namely, if we replace $S_i(\cdot|\cdot)$ by $S'_i(\cdot|\cdot)$, where $S'_i(r|\cdot)=S_i(\xi r|\cdot)$, with $\xi>0$ any scaling parameter.
	\end{itemize}
\end{observacao}

We call $\cX_k:=(X_1,\ldots,X_k)$ the {\em cascading jump evolution induced by}  $\ss_k$, denoted by CJE($\ss_k$).

$C_\cdot(\cdot|\cdot)$ defined in the sentence right above~\eqref{x1} and in~\eqref{clock} are called the {\em clock processes} associated to $\cX_k$.
And $\cK_k=\cK(\ss_k):=\{K_j(\cdot|\bar s_{j-1}),\,\bar s_j\in\ct_j,\,1\leq j\leq k\}$ are called the (family of) {\em clock parts} associated to $\ss_k$.
By the third item in Remark~\ref{X-process}, not every clock part enters a clock process.

\smallskip

The jump sizes of the functions of $\ss_k$ play the role of  state space,  and are related to the jump rates of $\cX_k$.
We will shortly adopt an alternative representation of $\cX_k$,  prompted by an elementary property of exponential distributions to be discussed in the next subsubsection, where the latter relationship is direct. 

\subsubsection{Alternative representation} \label{alt}

{One reason to consider an alternative representation for $\cX_k$ is that, while the jump sizes of the jump functions $S_\cdot(\cdot|\cdot)$ give the jump rates of the clock processes entering the description of $\cX_k$,
the actual jump rates of say $X_2$ within a(n exponentially distributed) sojourn interval of $X_1$ are affected by the size of that interval (since, in contrast to the corresponding clock process $C_2(\cdot|\cdot)$, $X_2$ jumps at the end of that interval, thus incrementing those rates in comparison). We will thus use the actual jump rates of $\cX_k$ to represent it. This is a perhaps more natural alternative, which has the technical advantage of yielding smallness of the coordinate at a certain level as soon as we have smallness of a previous level, even when the jump size of $S_\cdot(\cdot,\cdot)$ at the former level is not small. We will come back to the latter point in Remark~\ref{small} below.
}

{The building block for the new representation is} the following general result for the exponential distribution.
\begin{lema}\label{LemmaExercicio}
	Let $S\in\hat{\G}$ and consider $\hat{\T}_S=\{\gamma(s), \, s\in\T_S\}$,  the jump sizes of $S$. 
	Let $\{E,E_s; s\in\T_S\}$ be iid ~standard exponential random variables.  Let $\Lambda\in (0,\infty)$,  and  set $L=\Lambda^{-1}E$.  For $ r>0$, let 
	\begin{equation}
		K(r)=
		\sum_{s\in[0, r]}\E_s,\quad \text{with}\quad \E_s= \gamma(s)E_s,
	\end{equation}
	and make
	$R=\inf\{r>0:K(r)\geq L\}$.
	Then
	\begin{itemize}
		\item[(a)] $R\in\T_S$ almost surely;
		\item[(b)] Let  $\E=\{\E_s, s<R;\,\hat{\E}_R;\,\tilde{\E}_s, s>R\}$, where
		$\hat{\E}_R=L-K(R-)$,
		%
		and $\tilde{\E}:=\{\tilde{\E}_s,s\in\T_S\}$ is an iid ~family of random variables such that $\tilde{\E}_s\sim \exp(\La+{\lambda}(s))$,  
		with $\lambda(s)=(\gamma(s))^{-1}$, $s\in\T_S$.
		Then $\E\sim\tilde{\E}$;
		\item[(c)]
		Moreover, $\E$ and $R$ are independent, and 
		\begin{equation}\label{R}
			\pr(R>r)=\prod_{s\leq r}\frac{\l(s)}{\Lambda+\l(s)},\quad r\geq 0.
		\end{equation}
		
		\item[(d)] Let $\check{\mathcal{E}}_R$ denote $K(R)-L$;
		then, given $R$, $\check{\mathcal{E}}_R\sim \exp(\lambda(R))$ independently of $\hat{\E}_R$.
	\end{itemize}
\end{lema}

This is possibly well known; certainly when $S$ is discrete, but the general case works just as well; we provide a justification in Appendix~\ref{app1}, just in case.

\paragraph{Alternative construction of $\cX_k$.} \mbox{}
It is enough to obtain the labeled collection of lengths of constancy intervals of $X_1$ associated to each jump of $S_1$, for the first level, and then, inductively, for level $2\leq j\leq k$, within each constancy interval of $X_{j-1}$, say the one associated to $\bar s_{j-1}$,  to obtain the labeled collection of lengths of constancy intervals of $X_j$ associated to each jump of $S_j(\cdot|\bar s_{j-1})$. With the labeled collections of lengths in hands,  we inductively obtain, level by level, the constancy intervals of a given level $j$ within a constancy interval of the previous levels, say associated to $\bar s_{j-1}$, by  concatenation, following the order of the jumps of $S_j(\cdot|\bar s_{j-1})$. 

To give an explicit (alternative) construction, we start with  a precise description of the collections of jumps/labels entering each stage outlined in the previous paragraph. This is provided by Lemma \ref{LemmaExercicio}, as follows. By first taking $S=S_1$ in the lemma, we get $R_1$ distributed as in~\eqref{R}, and make $\cR_1=\{s\in\T_1: s\leq R_1\}$; these are the level-1 jumps of the process (in $[0,L]$). Proceeding inductively for $j\geq2$, given $\bar s_{i}\in\cR_{i}(\bar s_{i-1})$, $1\leq i\leq j-1$ ---with $\cR_{1}(\bar s_{0}):=\cR_{1}$---, we get from the lemma, with $S_j(\cdot|\bar s_{j-1})$ playing the role of $S$, $R_j(\bar s_{j-1})$ distributed as in~\eqref{R}, and make $\cR_j(\bar s_{j-1})=\{s\in\T_j(\bar s_{j-1}): s\leq R_j(\bar s_{j-1})\}$. We thus obtain the (cascading) family of jump labels 
$\bar\cR_1:=\cR_1, \bar\cR_2:=\bigcup_{s_1\in\bar\cR_1}\{s_1\}\times\cR_2(s_1),\ldots,
\bar\cR_k:=\bigcup_{\bar s_{k-1}\in\bar\cR_{k-1}}\{\bar s_{k-1}\}\times\cR_k(\bar s_{k-1})$.

Finally, in order to obtain the lengths of the constancy intervals of $\cX_k$, we proceed backwards. Let $E^k_{\bar s_k},\,\bar s_k\in\cR_k$, be an iid family of standard exponential random variables. The length of the constancy interval of $X_k$ associated to $\bar s_k\in\cR_k$ is given by
$\E^k_{\bar s_k}:=(\La+\La_k(s_k|\bar s_{k-1}))^{-1}E^k_{\bar s_k}$, in consistence  with Lemma  \ref{LemmaExercicio}.
For $j=2,\ldots,k$, and $\bar s_j\in\bar\cR_j$, having obtained the length of the constancy intervals of $X_j$ associated to $\bar s_j$, we have
that the length of the constancy interval of $X_{j-1}$ associated to $\bar s_{j-1}\in\bar\cR_{j-1}$ is given by 
$\E^{j-1}_{\bar s_{j-1}}:=\sum_{s_j\in\cR_j(\bar s_{j-1}) }\E^{j}_{\bar s_{j}}$, and we are done. 

It follows from Lemma \ref{LemmaExercicio} that $\E^{j}_{\bar s_{j}}\sim\exp(\La+\La_j(s_j|\bar s_{j-1}))$ for each $1\leq j\leq k$ and 
$\bar s_j\in\bar\cR_j$.

We can extend the alternative construction for all time (and not just $t\leq L$)  by repeating the above construction, but starting with $\bar\cR_1=\T_1$ instead, and then proceeding in the exact same way, thus obtaining the lengths of all the constancy intervals of $X_j$, $1\leq j\leq k$, along their full trajectory. In this case, we have that $\E^{j}_{\bar s_{j}}\sim\exp(\La_k(s_j|\bar s_{j-1}))$ for each $1\leq j\leq k$ and 
$\bar s_j\in\bar\cR_j$. This prompts the following alternative representation of $\cX_k$, which we adopt from now on:
\begin{equation}\label{Z}
	\cz_k=(Z_1,\ldots,Z_k),\,\text{ where }\, Z_j(t)=\big(\sum_{i=1}^jX_i(t)\big)^{-1},\,1\leq j\leq k.
\end{equation}
Note that for $1\leq j\leq k$, while in a constancy interval associated to $\bar s_j\in\bar\cR_j$ of either $X_j$ or $Z_j$ (it is the same interval), the mean waiting time until a jump occurs in either process is given by the value of $Z_j$, so it is in this sense natural to use $Z_j$ rather than $X_j$. Another reason for that is that the alternative representation turns out to be convenient for technical reasons, as will become clear in the next subsection.
Notice that there is a one-to-one correspondence between $\cz_k$ and $\cX_k$.

Also for technical reasons, it is convenient to consider the process in a finite exponential time interval $L$ as above, and that is what we will do below. In this case, we have that, while in a constancy interval of $Z_j$ associated to $\bar s_j\in\bar\cR_j$, $Z_j=\Ga_{j}(s_j|\bar{s}_{j-1})$, even though the mean waiting time until a jump occurs in $Z_j$ equals $\big[\Lambda+\big(\Ga_{j}(s_j|\bar{s}_{j-1})\big)^{-1}\big]^{-1}$ 
---recall the notation introduced in~\eqref{deflambda}.

\begin{observacao}\label{Pi}
	From the points made above, 
	we conclude that the left hand side of~\eqref{pip} equals
	\begin{equation}\label{age_lim}
		\Er(e^{-t/\hat Z^{(n)}_j(t_w)})	\to \Er(e^{-t/\hat Z_j(t_w)})	\,\text{ as }\,n\to\infty,
	\end{equation}
	for almost every realization of $\tau_\cdot^\cdot$, 
	where the convergence follows immediately from Theorem~\ref{teo:pserg}, and a similar conclusion holds for the left hand side of~\eqref{pi1} , with the latter expectation also showing up in the (double) limit, this time by an invocation of Theorem~\ref{teo:1serg}. 
	
	From the same points, we get that that expectation equals the expression in~\eqref{pi}. Now, well known scaling properties of $\a$-stable subordinators with $\a\in(0,1)$ imply self-similarity  of $\cz_k$, out of which it readily follows that the same expectation is a function of the quotient $t/t_w$. It in turn follows that the function is given by~\eqref{pii}, and other well known properties of stable subordinators and independence yield a more explicit expression for that function, as indicated in the paragraph below~\eqref{pii}. We return to this in more detail in  Appendix~\ref{app3}.	
\end{observacao}
{\begin{observacao}\label{small}
From~\eqref{Z}, the last point of the first paragraph of this subsubsection follows: as soon as $Z_i(\cdot)$ is small for some $i$, we have that $Z_j(\cdot)$ is at least as small for $j>i$. This will come in handy in dealing with convergence in Skorohod space. The alternative construction explained above, enabling the neat separation of jump labels and lengths of sojourn/constancy intervals, allowing for the backward procedure explained above, will also help in that respect.
\end{observacao}}

\subsection{Continuity theorem} \label{contheo}

For $k\geq 1$ and  $n\in\N$, let $~\ss_k^{(n)}$ be  a sequence of $k$-CFJF's  and  let   $\cz_k^{(n)}$ denote the Cascading Jump Evolution induced by $~\ss_k^{(n)}$.
We will prove a weak convergence result for $\tilde{Z}_k^{(n)}$ under the $J_1$ Skorohod topology on $D(\R_+^k)$
assuming $~\ss_k^{(n)}$ converges suitably to a $k$-CFJF.

It turns out that in the application of our abstract convergence theorem to the BDTM,  only dense or discrete  $~\ss_k$'s show up, so, for simplicity,  we will consider these cases only. 

Let $k\geq 1$ be fixed,  and let  $~\ss_k$ and $~\ss_k^{(n)}$, $n\in\N$,  be $k$-CFJF's and assume $~\ss_k$ is either dense or discrete.  
We  introduce a notion of convergence of  $\ss_k^{(n)}$ to$~\ss_k$ 
as follows.  We start by requiring  that 
\begin{equation}\label{cond1}
	S_1^{(n)}\to S_1,~~\text{on }~~ D(\R_+)~~ \text{as}\quad n\to\infty.
\end{equation} 

\begin{observacao}\label{ObserSaltos}
	A characteristic property of the $J_1$ topology (see \cite{WW}, Chapter 3, Section 3.3) is that if $x\in D$ has a jump at $t$, i.e., if $x(t)\neq x(t-)$ and if $d_{J_1}(x_n,x)\to 0$ in $D$, then there exists a sequence $(t_n)$ such $t_n\to t$, $x_n(t_n)\to x(t)$ and $x_n(t_n)-x_n(t_n-)\to x(t)-x(t-)$ as $n\to\infty$.
	In this context, we refer to $(t_n)$ as a $(x_n,x)$-{\em matching sequence} (or simply matching sequence, when the prefix is clear) to $t$.
\end{observacao}

\paragraph{Gap condition.} Additionally, if $S_1$ {\em is discrete}, we also require the following {\em gap condition} on the convergence in~\eqref{cond1}. 
Given $S\in\hat\G$, by a {\em gap} of $\T_S$, we mean any interval $(r,s)$ such that $r,t\in\T_S$, $r<t$, and $(r,t)\cap \T_S=\emptyset$. 
In the case where $S_1$ is discrete in~\eqref{cond1}, and that $(r,s)$ is a gap of $\T_1$, we require that, given two matching sequences $(r_n)$ and $(t_n)$ to $r$ and $t$ respectively, we have that for all  $n$ large enough $(r_n, t_n)\cap \T^{(n)}_1=\emptyset$, where  $\T^{(n)}_1:=\T_{S_1^{(n)}}$.

For the next levels ---assuming $k\geq2$---, we require that for each $\bar s_{k-1}\in\ct_{k-1}$ and $n\geq1$, 
there exists $\bar s^{(n)}_{k-1}=\bar s^{(n)}_{k-1}(\bar s_{k-1})\in\ct^{(n)}_{k-1}$, 
such that
\begin{itemize}
	\item  $s^{(n)}_{j}(\bar r_{k-1})=s^{(n)}_{j}(\bar t_{k-1})$ whenever $\bar r_{k-1},\,\bar t_{k-1}\in\ct_{k-1}$ are such that 
	$\bar r_j=\bar t_j$ for some $1\leq j<k$; 
	\item  For all $1\leq j<k$, we have that $(s^{(n)}_j)$ is an 
	$(S_{j}^{(n)}(\cdot|\bar{s}_{j-1}^{(n)}),S_{j}(\cdot|\bar{s}_{j-1}))$-matching sequences to $s_j$, 
	and 
	\begin{equation}\label{3bul}
		S_{j+1}^{(n)}(\cdot|\bar{s}_j^{(n)})\to S_{j+1}(\cdot|\bar{s}_j)\text{  as } n\to\infty;
	\end{equation}
	\item In the discrete case, we additionally require the gap condition on the latter convergence.
\end{itemize}

Under these conditions, we say that
\begin{equation}\label{conv}
	\ss_k^{(n)}\to \ss_k\text{  as } n\to\infty.
\end{equation}

\begin{teo}\label{PrinciTeor}
	Let $\ss_k^{(n)}$, $n\geq 1$, and $\ss_k$ be CFJF's,  with $\ss_k$ either dense or discrete,  such that $\ss_k^{(n)}\to\ss_k$ as $n\to\infty$, as above.  Let $\cz_k^{(n)}$ and $\cz_k$ be the Cascading Jump Evolutions  induced  by $\ss_k^{(n)}$ and $\ss_k$, respectively. Then $\cz_k^{(n)}\rightarrow \cz_k$ in distribution in  Skorohod space as $n\to\infty$.
\end{teo}

The weak nature of the convergence in  Theorem  \ref{PrinciTeor} comes from the fact that the families of exponential random variables used in the construction of $\tilde{Z}_k^{(n)}$ and $\tilde{Z}_k$ are not necessarily coupled. We find it convenient to prove a stronger result, and for  this reason we will take a coupled assignment of exponential random variables to the matched jumps of $\ss_k^{(n)}$ and $\ss_k$, so that we obtain indeed convergence in probability in the coupled space. 

We proceed to the coupling setup next.

\subsubsection{Coupling of the exponential random variables of $\cK_k$ and $\cK_k^{(n)}$}\label{SecCoupling}
We will couple the standard exponential random variables in the clock parts associated to $\ss_k$ and $\ss_k^{(n)}$, recall the notation introduced in the second paragraph below Remark~\ref{X-process}, and set $\cK_k^{(n)}=\cK(\ss_k^{(n)})$.

Two independent families of iid standard exponential random variables will come into play, namely
\begin{align}
	\fe_j&=\{E_j(s_j|\bar s_{j-1}),\,\bar s_j \in\ct_j\},\,1\leq j\leq k\; ;
    \\
	\fen_j&=\{E_j^{(n)}(s_j|\bar s_{j-1}),\,\bar s_j \in\ct^{(n)}_j\},\,1\leq j\leq k,\,n\geq1.
\end{align}
We use the first family in $\cK_k$, as described in~\eqref{K}, and in $\cK_k^{(n)}$ we will use some exponential variables of the first family, and the remainder variables from the second family, as explained in detail below.

Set $\check\T_k=\{\ga_j(s_j|\bar s_{j-1}),\,\bar s_j\in\ct_j,1\leq j\leq k\}$.
For $n\geq1$, let  $\ct_k^{(n)}:=\ct_k(\ss_k^{(n)})$, and set 
$\check\T_k^{(n)}=\{\ga^{(n)}_j(s^{(n)}_j|\bar s^{(n)}_{j-1}),\,\bar s^{(n)}_j\in\ct^{(n)}_j,1\leq j\leq k\}$ ---recall Definition~\ref{defCasFunct}.

Since $\ct_k$, $\ct_k^{(n)}$, 
$n\geq1$, are countable,  we may find a strictly decreasing sequence of real numbers $(a_l)_{l\geq 1}$ such that $a_l\to 0$ as $l\to\infty$ and 
\begin{equation}\label{a}
	\left\{	\bigcup_{n}\check{\T}_k^{(n)}\cup \check{\T}_k\right\}\cap \{a_l,l\geq 1\}=\emptyset
\end{equation}

We now denote the  set of {\em compounded matching sequences} guaranteeing the convergence in~\eqref{conv} 
as follows. Let
\begin{widetext}
\begin{equation*}
	\vec\ct_k=\bigcup_{\bar s_{k-1}\in\ct_{k-1}}\big\{(\bar s_{k-1}^{(n)}) \text{ satisfying the three bullet items of the paragraph of~\eqref{3bul}}\big\}.
\end{equation*}
\end{widetext}
{Let us note that in each subset of the above union there may be several (even infinitely many) compounded matching sequences $(\bar s_{k-1}^{(n)})$ for each $\bar s_{k-1}\in\ct_{k-1}$. However, for the coupling argument we undertake in this subsubsection, and its subsequent application to our convergence results, it is convenient to deal with a single compounded matching sequence for each $\bar s_{k-1}\in\ct_{k-1}$. We thus pick,}
for each $\bar s_{k-1}\in\ct_{k-1}$, a {\em single} corresponding compounded matching sequence 
$(\bar t_{k-1}^{(n)})$ to $\bar s_{k-1}$ from $\vec\ct_k$, and let $\vtp$ 
denote the collection of such compounded matching sequences, so that
\begin{widetext}
\begin{equation*}
	\vtp=\bigcup_{\bar s_{k-1}\in\ct_{k-1}}\big\{(\bar t_{k-1}^{(n)}) \text{ satisfying the three bullet items of the paragraph of~\eqref{3bul}}\big\},
\end{equation*}
\end{widetext}
{so that we now have each subset in the above union containing a single element}. In order to complete the compounded matching sequence with a $k$-th coordinate, we pick, for each 
$\bar s_k$, $(\bar t_{k-1}^{(n)})=(\bar t_{k-1}^{(n)}(\bar s_{k-1}))\in\vtp$, and choose a single $(S_{k}^{(n)}(\cdot|\bar{t}_{j-1}^{(n)}),S_{k}(\cdot|\bar{s}_{k-1}))$-matching sequence $(t_{k}^{(n)})$ to $s_k$, thus obtaining 
$(\bar t_{k}^{(n)})=(\bar t_{k}^{(n)}(\bar s_{k}))$. 
We let
\begin{equation}\label{vtpp}
	\vtpp=\bigcup_{\bar s_{k}\in\ct_{k}}\big\{(\bar t_{k}^{(n)})\} .
\end{equation}
{We again have each subset in the above union containing a single element}.
\begin{observacao}\label{cons}
	For each $n\geq1$, $1\leq j\leq k$, and $\bar s_k\in\ct_k$,  the first $j$ coordinates of the $n$-th entry of $(\bar t_{k}^{(n)})$, namely 
	$\bar t_{j}^{(n)}(\bar s_k)$ depends only $\bar s_j$, that is, $\bar t_{j}^{(n)}(\bar s_k)=\bar t_{j}^{(n)}(\bar s_j)$. This follows from the first bullet item
	in the paragraph of~\eqref{3bul}.
\end{observacao}

\begin{observacao}\label{just}
	Again, taking single choices of compounded matching sequences as in~\eqref{vtpp} is convenient for the coupling procedure we describe below, as well as in the application to the BDTM in the next section, where, since the convergences in that context hold almost surely (after we, resorting to Skorohod Representation, move to a suitable probability space), it facilitates  control of exceptional events.
\end{observacao}

Given $\bar s_k\in\ct_k$, we define $\T_{j,l}(\bar s_{j-1})=\{r\in\T_j(\bar s_{j-1}):\,\Gamma_j(r|\bar s_{j-1})\in (a_l,a_{l-1})\}$, $1\leq j\leq k$, $l\geq 1$, a set that may be enumerated in increasing order as $\T_{j,l}(\bar s_{j-1})=\{r_1(j,l)<r_2(j,l)<\cdots\}$, where $r_m(j,l)=r_m(j,l;\bar s_{j-1})$, 
$m\geq1$ \footnote{and may be finite, and even empty, for some $l\geq1$}, 
and let $l_1,m_1\ldots,l_k,m_k$ be such that $s_j=r_{m_j}(j,l_j)=:r_j=r_j(\bar s_j)$, $1\leq j\leq k$.

Now, given 
$(\bar t_{k}^{(n)})\in\vtpp$ associated to $\bar s_k$, we proceed with the coupling of the exponential random variables
from $K_j(\cdot|\bar s_{j-1})$ at the jump in $s_j$ and from $K^{(n)}_j(\cdot|\bar t^{(n)}_{j-1})$ at the jump in $t^{(n)}_{j}$, $1\leq j\leq k$, as follows.
For fixed $j=1,\ldots,k$ and $n\geq1$, if for $i=1,\ldots, j$
\begin{align}
	 &\Gamma_i^{(n)}(t_i^{(n)}|\bar t_{i-1}^{(n)})\in(a_{l_i},a_{l_i-1}), \,\nonumber\text{ and }\\
	&|t_i^{(n)}-s_i|<|t_i^{(n)}-r_{m}(i,l_i)|\, \text{ for all } \, m\geq 1,  m\neq m_i,\label{jumpcontitionk} 
\end{align}
then we assign $E_j(s_j|\bar{s}_{j-1})$ from $\fe_j$ as the standard exponential random variable in the jump of $K_j^{(n)}(\cdot|\bar{t}_{j-1}^{(n)})$ 
at $t^{(n)}_{j,i}(l)$;
otherwise, we use $E_j^{(n)}(t_j^{(n)}|\bar t_{j-1}^{(n)})$ from $\fen_j$, instead; assignments to $K_j^{(n)}(\cdot|\bar{s}_{j-1}^{(n)})$ with 
$\bar{s}_{j-1}^{(n)}\ne\bar{t}_{j-1}^{(n)}$ are also taken from $\fen_j$.

\begin{observacao}\label{coup1}\mbox{ }
	For each $j=1,\ldots,k$, $\bar s_j\in\ct_j$ and $n\geq1$,~\eqref{jumpcontitionk} holds for at most one choice of $\bar s^{(n)}_j\in\ct_j^{(n)}$, 
	namely, $\bar t_j^{(n)}$, the first $j$ coordinates of $\bar t_k^{(n)}=\bar t_k^{(n)}(\bar s_k)$, the $n$-th entry of $(\bar t_k^{(n)})\in\vtpp$ corresponding to $\bar s_k$, and there exists $n_0=n_0(\bar s_j)$ such that it does for that choice and  $n\geq n_0$.
\end{observacao}

We denote the family of standard exponential random variables assigned to $K^{(n)}_j(\cdot|\cdot)$, $1\leq j\leq k,\,n\geq1$, as above by
\begin{align}
	\fent_j=\{\hat E_j^{(n)}(s_j|\bar s_{j-1}),\,\bar s_j \in\ct^{(n)}_j\},\,1\leq j\leq k,\,n\geq1.
\end{align}
This will  be the standard exponential random variables in the construction of  the clock parts and clock processes in the definition of $\cz_k^{(n)}$, $n\geq1$.

\subsubsection{Auxiliary results}

Given $\ell\geq1$ and $\Lambda>0$, let 
\begin{equation}
	\cQ=\cQ_\ell=\{s\in\cR_1: \, 
	\gamma_1(s)>a_\ell\},
\end{equation}
where $(a_\ell)$ was given in the paragraph of~\eqref{a} above, and $\cR_1$ was defined in the alternative construction of $\cX_k$ in Subsubsection~\ref{alt} above.
Assuming $\cQ\ne\emptyset$, we may  almost surely write $\cQ=\{q_1<\ldots<q_{N}\}$ for some $1\leq N<\infty$, 
and we take $N=0$ to mean that $\cQ=\emptyset$. 

Assuming $N\geq1$, for $i=1,\ldots,N$, let $(q_i^{(n)})$ be the matching sequence to $q_i$ obtained from $\vtpp$ ---i.e., $q_i^{(n)}=t_1^{(n)}(q_i)$.  Let $\delta=\delta_\ell>0$ be such that $\gamma_1(s)<a_\ell$ for all  $s\in(R_1,R_1+\delta)$. Notice that $\delta$ is almost surely well defined, random, and $\delta_\ell\to0$ as $\ell\to\infty$ almost surely. Let us now write 
\begin{align}\label{cqn}
	\cQ^{(n)}&=\{s\in\T_1^{(n)}: s<R_1+\delta \,\text{ and }\, \gamma_1^{(n)}(s)>a_\ell\}\nonumber\\
    &=\{s_1^{(n)}<\ldots<s_{M_n}^{(n)}\}.
\end{align}

\begin{lema}\label{LemaPre1}
	Almost surely, for all $n$ large enough, we have that 
	$$M_n=N\,\text{ and  }\, \cQ^{(n)}=\{q_1^{(n)}<\ldots<q_N^{(n)}\}=:\tilde{\cQ}^{(n)}. $$
\end{lema}

\begin{proof}
	Given $\ell$ we have that 
	\begin{enumerate}
		\item[(1)] For all $n$ large $\{q_1^{(n)}<\ldots<q_N^{(n)}\}\subset \cQ^{(n)}$. This follows immediately from the definition of $\cQ^{(n)}$ and the fact that $(q_i^{(n)})_n$ is a matching sequence to $q_i$, $i=1,\ldots,N$.
		\item[(2)] If $\{q_1^{(n)}<\ldots<q_N^{(n)}\}\ne \cQ^{(n)}$ infinitely often, then there must exist a subsequence $(s^{(n')}:=s^{(n')}_{i_{n'}})$, with $1\leq i_{n'}\leq M_{n'}$, such that $s^{(n')}\notin \tilde{\cQ}^{(n)}$, $s^{(n')}\to s$ as $n'\to\infty$, and $s\in[0,R_1+\delta)$, $\gamma_1^{(n')}(s^{(n')})\to\gamma_1(s)\geq a_\ell$ as $n'\to\infty$. Then $s\in\T_1$, $s\leq R_1$, $\gamma_1(s)>a_\ell$ \footnote{Equality here is not allowed by~\eqref{a}.}, and we must have that $s=q_i$ for some $i=1,\ldots, N$, which means that  $s^{(n')}$ is a matching sequence for $q_i$ and $s^{(n')}\neq q_i^{(n')}$ for every $n'$, which contradicts  the convergence $S_1^{(n)}\to S_1$ as $n\to\infty$ in the $J_1$ metric.
	\end{enumerate}
\end{proof}
In the notation of Lemma \ref{LemaPre1}, let $\cI:=\{I_1,\ldots,I_N\}$ denote the constancy intervals of ${Z}_1$ in $[0,L]$, associated to $\cQ=\{q_1,\ldots,q_N\}$, respectively, and write $I_m=[b_m,d_m)$, $1\leq m\leq N$.
Similarly, we denote by $\{I_1^{(n)},\ldots,I_{M_n}^{(n)}\}$  the set of constancy intervals of ${Z}_1^{(n)}$  in $[0,L_n]$ associated to $s\in\T_1^{(n)}$ such that $\gamma_1^{(n)}(s)>a_\ell$, where $L_n=\Lambda_n^{-1}T$, with $\Lambda_n\to\Lambda$ as $n\to\infty$, and write $I_m=[b^{(n)}_m,d^{(n)}_m)$, in increasing order in $m$, $1\leq m\leq M_n$.
\begin{lema}\label{LemaPre2}
	Given $\varepsilon,\varepsilon'>0$, there exists $\ell=\ell_\ve$ and $\ve'=\ve'_\ve$, with  $\ve'\to0$ as $\ve\to0$ such that, with probability bounded below by $ 1-c_\varepsilon$ for all $n$ large,  with $\ce\to0$ as $\ve\to0$, we have that:
	\begin{itemize}
		\item[(i)] $N\geq1$ and $q_N=R_1$;
		\item[(ii)] $M_n=N$,  $\{I_1^{(n)},\ldots,I_N^{(n)}\}$ correspond to constancy intervals of $Z_1^{(n)}$ associated to $\{q_1^{(n)},\ldots,q_N^{(n)}\}$,  respectively, and 
		\item[(iii)] $\max_{1\leq m\leq N}|b^{(n)}_m-b_m|\leq\ve'$; \,$\max_{1\leq m\leq N}|d^{(n)}_m-d_m|\leq\varepsilon'$. 
	\end{itemize}
\end{lema}
\begin{proof}
	Item (i) does not depend on $n$; the first assertion is quite clear, and the second follows from Lemma~\ref{LemmaExercicio}. For the remaining items, we argue as follows.
	\begin{enumerate}
		\item[(1)] By Markov's inequality,  for all $A\geq 0$ we have that 
        \begin{widetext}
		\begin{align} 
			\pr\left(\sum_{s\in\T_1; s\leq R_1\atop \gamma_1(s)<a_\ell} \gamma_1(s)E_1(s)>\varepsilon\right)
			&=\pr\left(\sum_{s\in\T_1; s\leq R_1\atop \gamma_1(s)<a_\ell} \gamma_1(s)E_1(s)>\varepsilon;\,R_1\leq A\right)  \nonumber
			+\pr\left(\sum_{s\in\T_1; s\leq R_1\atop \gamma_1(s)<a_\ell} \gamma_1(s)E_1(s)>\varepsilon;\,R_1> A\right) \nonumber \\
			&\leq\frac{1}{\varepsilon}      \underbrace{\displaystyle\sum_{s\in\T_1; s\leq A\atop \gamma_1(s)<a_\ell}
				\gamma_1(s)}_{\kappa}+\pr(R_1>A) \label{MarkIneq1}
		\end{align}
        \end{widetext}
		Given $\varepsilon>0$, we may choose $A$ and $\ell$ such that $\pr(R_1>A)<\ve$ 
		and $\kappa\leq \ve^2$; 
		it follows from this choice that  the right hand side of (\ref{MarkIneq1}) is bounded above by $2\ve$. 
		\item[(2)] Similarly, 
		\begin{align} 
		&\pr\left(\displaystyle\sum_{s\in\T_1^{(n)};\,
				s< R_1+\delta\atop\gamma_1^{(n)}(s)<a_\ell}\gamma_1^{(n)}(s)\hat E^{(n)}_1(s)>\varepsilon\right) \nonumber \\
			&\leq\frac{1}{\varepsilon}
			\underbrace{\displaystyle\sum_{s\in\T_1^{(n)};\,
					s\leq A\atop\gamma_1^{(n)}(s)<a_\ell}\gamma_1^{(n)}(s)}_{\kappa'}+\pr(R_1+\delta>A)\label{MarkIneq2}
		\end{align}
		Since $(\delta_\ell)$ is tight, we may choose $A$ such that $\pr(R_1+\delta>A)<\ve$ 
		for all $\ell\geq1$. 
		Taking $\ell$ as at the end of (1), by the convergence $S_1^{(n)}\to S_1$, we have that, for all large $n$, $\kappa'\leq\ve^2$, 
		and thus the right hand side of (\ref{MarkIneq2}) is bounded above by $2\ve$. 
		
		\item[(3)] Let us choose $\ell$ so large that $q_{N}=R_1$ with probability larger than $1-\ve$.
		From (1), enlarging $\ell$ if necessary, 
		we have that for $1\leq m\leq N$, with probability larger than $1-3\ve$
		\begin{align}
			&\sum_{i=1}^{m-1}\gamma_1(q_i)E_1(q_i)\nonumber\\
			&\geq K_1(q_m-)-\sum_{\substack{s\in\T_1:\,s< R_1;\\\gamma_1(s)<a_\ell}}\gamma_1(s)E_1(s) \nonumber\\
			&\geq K_1(q_m-)-\varepsilon,
		\end{align}
        where $\sum_{i=1}^0 \cdots = 0$.
		%
		 
		Thus, using the coupling, for $1\leq m\leq N$ and 
		all large $n$ 
		\begin{align}
			&\sum_{i=1}^{m-1}\gamma_1^{(n)}(q_i^{(n)})\hat E^{(n)}_1(q_i)\nonumber\\
            &=\sum_{i=1}^{m-1}\gamma_1^{(n)}(q_i^{(n)})E_1(q_i)\nonumber \\
            &\geq K_1(q_m-)-2\varepsilon=b_m-2\ve \label{lc1}
		\end{align}
		with probability larger than $1-3\ve$. 
		
		Now, outside an event of probability smaller than $\ve$, 
		we may assume that the conclusion of  Lemma \ref{LemaPre1} holds. Using this and the fact that
		$\gamma_1^{(n)}(q_m^{(n)})E_1(q_m)=K_1^{(n)}(q_m^{(n)})-K_1^{(n)}(q_m^{(n)}-)$ for $1\leq m\leq N$ and
		all large $n$ (by the coupling of Subsubsection~\ref{coup1}),  
		we have that 
		\begin{align}
			&\sum_{\substack{s\in\T_1^{(n)}:\, s\leq q^{(n)}_m\wedge(R_1+\delta);\\\gamma_1^{(n)}(s)>a_\ell}}\!\!\!\!\!\!\!\!\!
			\gamma_1^{(n)}(s)\hat E_1^{(n)}(s) \nonumber\\
			&=\sum_{i=1}^{m-1}\gamma_1^{(n)}(q_i^{(n)})E_1(q_i) +\gamma_1^{(n)}(q_m^{(n)})E_1(q_m)  \nonumber \\
			&\geq K_1(q_m)-3\varepsilon. 
		\end{align}
		with probability larger than $1-4\ve$. 
		At this point, let us remark that by Lemma \ref{LemmaExercicio}, we have that 
		with probability larger than $1-\ce'$, with $\ce'\to0$ as $\ve\to0$, $K_1(R_1)-L>3\ve$.
		Thus, enlarging $n$ if necessary, we have that
		\begin{align}\label{lln}
			&\sum_{\substack{s\in\T_1^{(n)}:\,
					s\leq R_1+\delta}}\gamma_1^{(n)}(s)\hat E_1^{(n)}(s)\\
                    &\geq	\sum_{\substack{s\in\T_1^{(n)}:\,s\leq R_1+\delta;\\\gamma_1^{(n)}(s)>a_\ell}}
			\gamma_1^{(n)}(s)\hat E_1^{(n)}(s)\nonumber\\
            &\geq L_n-3\ve.
		\end{align}
		with probability larger than $1-4\ve-\ce'$. 
		
		Now,
		from the coupling,  it follows that for $1\leq m\leq N$ and all large $n$
		\begin{align}\label{sc1}
			&\sum_{i=1}^{m-1}\gamma_1^{(n)}(q_i^{(n)})\hat E_1^{(n)}(q_i^{(n)})\nonumber\\
            &\leq\sum_{i=1}^{m-1}\gamma_1(q_i)E_1(q_i)+\ve\nonumber\\
			&\leq b_m+\ve 
		\end{align}
		with probability bounded below by $1-\ve$;  
		it follows from (2) that 
		\begin{align}
			\tilde b^{(n)}_m\,\,:&=\sum_{\substack{s\in\T_1^{(n)}:\,s< q_m^{(n)} \wedge(R_1+\delta)}}\gamma_1^{(n)}(s)\hat E_1^{(n)}(s)\nonumber \\
			&\leq\sum_{\substack{s\in\T_1^{(n)}:\,s\leq R_1+\delta\\\gamma_1^{(n)}(s)<a_\ell}} 
			\gamma_1^{(n)}(s)\hat E_1^{(n)}(s)\nonumber\\
            & \hspace{1cm}+\sum_{i=1}^{m-1}\gamma_1^{(n)}(q_i^{(n)})\hat E_1^{(n)}(q_i^{(n)})\nonumber\\
			&\leq b_m+2\ve, \label{ineqsemtN}
		\end{align}
		with probability bounded below by $1-4\ve$ for $1\leq m\leq N$ and all large $n$. 
		It follows from~\eqref{ineqsemtN} that
		\begin{equation}\label{sln}
			\tilde b^{(n)}_N\leq L_n,
		\end{equation}
		as soon as $n$ is large enough and $L-K_1(R_1-)\geq 3\ve$, an event whose  probability we write as $1-c''_\ve$; by Lemma~\ref{LemmaExercicio},  $c''_\ve$ as $\ve\to0$.
	\end{enumerate}
	
	
	In the intersection of the events where inequalities in~\eqref{lc1},~\eqref{lln},~\eqref{ineqsemtN} and~\eqref{sln}  hold,  and where  the conclusion of  Lemma \ref{LemaPre1} holds, we have that $\tilde b^{(n)}_m= b^{(n)}_m\geq b_m-2\ve$ for $1\leq m\leq N$, and Item (ii) and the first claim of Item (iii) of the lemma then also hold, and that happens with probability larger than $1-11\ve-\ce'-\ce''$ for all large $n$. 
	
	As for the second claim of  Item (iii), it follows readily from the first claim once we disregard $m=N$ in the max sign, upon noticing
	that $d^{(n)}_m-b^{(n)}_m=\ga_1^{(n)}(q^{(n)}_m)E_1(q_m)$ and $d_m-b_m=\ga_1(q_m)E_1(q_m)$ for all large $n$
	from the coupling, and $\ga_1^{(n)}(q^{(n)}_m)\to \ga_1(q_m)$ as $n\to\infty$ 
	(once we are in the event where the conclusion of  Lemma \ref{LemaPre1} holds). This is perhaps less clear for the $m=N$ case, 
	which may be justified as follows. In the event where~\eqref{lln} and~\eqref{sln} take place, we have that $d^{(n)}_N-b^{(n)}_N=L_n-K_1^{(n)}(q_N^{(n)}-)$, as soon as $n$ is sufficiently large, so it is enough to argue the claim that
	\begin{equation}\label{epsl}
		|(L_n-K_1^{(n)}(q_N^{(n)}-))-(L-K_1(R_1-))|\leq\varepsilon'
	\end{equation}
	with probability larger than $1-\ce'''$, with $\ce'''\to0$ as $\ve\to0$, for all large $n$. 
	In the event where~\eqref{lc1}, \eqref{ineqsemtN}   and  the conclusion of  Lemma \ref{LemaPre1} hold, we have that
	\begin{align*}
	K_1(R_1-)-2\ve &\leq \sum_{i=1}^{N-1}\gamma_1^{(n)}(q_i^{(n)})\hat E_1^{(n)}(q_i^{(n)})\\
    &\leq K_1^{(n)}(q_N^{(n)}-)\leq K_1(R_1-)+2\ve, 
	\end{align*}
	so that the left hand side of~\eqref{epsl} is bounded above by $3\ve$ as soon as $n$ is large enough,
	and~\eqref{epsl} holds with probability larger than $1-\ce''':=1-11\ve-\ce'$ 
	for all large $n$, and the claim follows once $\ve'\geq 3\ve$.
\end{proof}

\subsubsection{Preliminaries to the proof of Theorem \ref{PrinciTeor}}

Let us develop some notation which  will enter  the proof of Theorem \ref{PrinciTeor}.
As in the setting of the previous subsubsection, we fix $\Lambda>0$ and a standard exponential random variable $T$, and make $L=\Lambda^{-1}T$. 
We fix $\ell_1,\ldots,\ell_k\geq1$ arbitrarily.
\begin{enumerate}
	\item[1)] 
	Write $\ell_1$, $\cQ_1$  and $N_1$ for  $\ell$, $\cQ$ and $N$ in the preliminaries to Lemma~\ref{LemaPre1}, respectively. Set $\mt_1=\cQ_1$, and let
	$\cI_1$ be as $\cI$ in Lemma~\ref{LemaPre2}.
	
	\item[2)]Proceeding inductively on the level, given $1<j\leq k$, and a constancy interval of $\bar{Z}_{j-1}:=(Z_1,\ldots,Z_{j-1})$  in $[0,L]$ associated to $\bar{s}_{j-1}\in\mt_{j-1}$ 
	---we denote it by $I_{j-1}(\bar{s}_{j-1})$---, 
	let 
	$$
	\cQ_{j}(\bar s_{j-1})=\{s\in\cR_j(\bar s_{j-1}):\,  
	\gamma_j(\bar s_{j-1})>a_{\ell_j}\},
	$$ 
	and set 
	$N_{j}(\bar s_{j-1})=|\cQ_{j}(\bar s_{j-1})|$, the cardinality of $\cQ_{j}(\bar s_{j-1})$. 
	Also, set $$\mt_{j}=\bigcup_{\bar r_{j-1}\in\mt_{j-1}}\{\bar r_{j-1}\}\times\cQ_{j}(\bar r_{j-1}).$$
	Here and below, $\{\cdot\}\times\emptyset=\emptyset$.
	Let now $$\cI_{j}(\bar s_{j-1})=\Big\{I_j(\bar s_j)=\big[b_{j}(\bar s_{j-1}),d_{j}(\bar s_{j-1})\big),\,s_j\in\cQ_{j}(\bar s_{j-1})\Big\}$$
	be  the collection of constancy intervals of $\bar{Z}_{j}$ in $I_j(\bar{s}_{j-1})$ associated to $\bar s_{j}\in\mt_j$. 
\end{enumerate}

Similarly,  we introduce the respective notation for the processes $\cz_k^{(n)}$, $n\geq1$, up to time $L$, obtaining, for $j=1,\ldots,k$, 
$$
\cQ^{(n)}_{j}(\bar s^{(n)}_{j-1})=\big\{s\in\cR^{(n)}_j(\bar s^{(n)}_{j-1}):\, \gamma^{(n)}_j(\bar s^{(n)}_{j-1})>a_{\ell_j}\big\},
$$
$N^{(n)}_{j}(\bar s^{(n)}_{j-1})=|\cQ^{(n)}_{j}(\bar s^{(n)}_{j-1})|$, 
$\mt^{(n)}_1=\cQ^{(n)}_1$ and
$\mt^{(n)}_{j}=\bigcup_{\bar r^{(n)}_{j-1}\in\mt^{(n)}_{j-1}}\{\bar r^{(n)}_{j-1}\}\times\cQ^{(n)}_{j}(\bar r^{(n)}_{j-1}),\,1<j\leq k$,
and
the collection 
\begin{widetext}
\begin{equation*}
\cI^{(n)}_{j}(\bar s_{j-1}^{(n)})=\Big\{I^{(n)}_j(\bar{s}^{(n)}_j)
=\big[b^{(n)}_j(\bar{s}^{(n)}_j), d^{(n)}_j(\bar{s}^{(n)}_j)\big),\,s^{(n)}_j\in\cQ^{(n)}_{j}(\bar s^{(n)}_{j-1})\Big\}
\end{equation*}
\end{widetext}
of constancy intervals of $\bar Z_j^{(n)}$ in $[0,L]$ associated to $\bar s_j^{(n)}\in\mt_j^{(n)}$. 

Let us now pick for each $\bar s_k\in\mt_k$ the corresponding compounded matching sequence $(\bar t_k^{(n)})=(\bar t_k^{(n)}(\bar s_k)$ from $\vtpp$ ---recall the definition in~\eqref{vtpp}---, and for each $n\geq1$, set $\tilde\mt^{(n)}_k=\bigcup_{\bar s_k\in\mt_k}\{\bar t_k^{(n)}(\bar s_k)\}$. Now for each $0\leq j<k$ and $n\geq1$, let $\tcq^{(n)}_{j+1}(\bar t^{(n)}_j)=\{ r_{j+1}:\,\bar r_k\in\tilde\mt^{(n)}_k\text{ and }\bar r_j=\bar t^{(n)}_j\}$; the first bullet item in the paragraph of~\eqref{3bul} ensures that this is uniquely defined.

Let the elements of $\mt^{(n)}_k$ and $\mt_k$ be put in lexicographic order.

\begin{coro}[Corollary to Lemma \ref{LemaPre2}]\label{Corotolema2}
	Given $\varepsilon,\varepsilon''>0$, $\ell\geq 0$ and 
	$\Lambda>0$, there exist (deterministic) $\hat N_1,\ldots,\hat N_k$ and $\ell_1\geq \ell, \ell_2\geq \ell,\ldots, \ell_k\geq l$ such that the event where
	\begin{align}\label{eqcortolema}
		\cQ^{(n)}_{j}(\bar s^{(n)}_{j-1})=
		\tcq^{(n)}_{j}(\bar t^{(n)}_{j-1}), &\quad
		1\leq N^{(n)}_{j}(\bar s^{(n)}_{j-1})=N_{j}(\bar s_{j-1}) \leq \hat N_j;\\\label{eqcortolema2}
		\Big|b^{(n)}_j(\bar{s}^{(n)}_j)-b_j(\bar{s}_j)\Big|\leq\varepsilon'', & \quad
		\Big|d^{(n)}_j(\bar{s}^{(n)}_j)-d_j(\bar{s}_j)\Big|\leq\varepsilon'', & \quad
	\end{align}
	$1\leq j\leq k,\,\bar s^{(n)}_k\in\mt^{(n)}_k,\,\bar s_k\in\mt_k$,
	has probability larger than $1-\d$ for all $n$ large enough, where $\d=\d_\ve$ is such that $\d\to0$ as $\ve\to0$. 
\end{coro}

\begin{observacao}\label{obsersubsequence}
	Notice that when~\eqref{eqcortolema} holds, we may and will make  $\bar{t}^{(n)}_k=\bar{s}^{(n)}_k$ for all $\bar{s}^{(n)}_k\in\mt^{(n)}_k$,
	thus establishing a 1 to 1 correspondence between $\bar{s}^{(n)}_k\in\mt^{(n)}_k$ and $\bar{s}_k\in\mt_k$.
\end{observacao}

\begin{proof}
	We use Lemma~\ref{LemaPre2} repeatedly with $\ve'=\frac{\ve''}{2k}$.
	Lemma~\ref{LemaPre2} and the fact that $N_1$ is a.s.~finite yield~\eqref{eqcortolema} for $j=1$  with probability larger than $1-
	\frac12\d$, with $\d:=2c_{1,\ve}:=2(c_\ve+\ve)$, by choosing $\ell_1=\ell$ in that lemma. We proceed inductively to levels $1<j\leq k$. Assuming it holds for $1\leq i<j$, $\bar s^{(n)}_{j-1}\in\mt^{(n)}_{j-1},\,\bar s_{j-1}\in\mt_{j-1}$, for suitable choices of $\ell_1,\ldots,\ell_{j-1}\geq\ell$ and
	$\hat N_1\,\ldots, \hat N_{j-1}$, with probability larger than $1-(1-2^{-j+1})\d$, we first notice that for all $\bar s_{j-1}\in\mt_{j-1}$, the constancy interval $I_{j-1}(\bar s_{j-1})$ and $I^{(n)}_j(\bar{s}^{(n)}_{j-1})=I^{(n)}_j(\bar{t}^{(n)}_{j-1})$, are, by Lemma~\ref{LemmaExercicio}, exponentially distributed, with respective rates $\La+\La_{j-1}(s_{j-1}|\bar s_{j-2})$ and $\La+\La^{(n)}_{j-1}(t^{(n)}_{j-1}|\bar t^{(n)}_{j-2})$, which converge one to the other as $n\to\infty$, so the condition of Lemma~\ref{LemaPre2} is satisfied. Notice that the number of such intervals is bounded by 
	$\check N_{j-1}:=\prod_{i=1}^{j-1}\hat N_i$. 
	
	Notice also, at this point, that as far as the lower ends of the constancy intervals 
	$\cI_{j}(\bar s_{j-1})$ and $\cI^{(n)}_{j}(\bar s_{j-1}^{(n)})$, namely  $b_j(\bar s_j)$, $\bar s_{j}\in\cQ_j(\bar s_{j-1})$,
	and $b^{(n)}_j(\bar s^{(n)}_j)$,  $\bar s^{(n)}_{j}\in\cQ^{(n)}_j(\bar s^{(n)}_{j-1})$, $1\leq j\leq k$, $\bar s_k\in\mt_k$, 
	$\bar s^{(n)}_k\in\mt^{(n)}_k$, are concerned, we have that Lemma~\ref{LemaPre2} applies to 
    \begin{align}\label{bdif}
        \tilde b_j(\bar s_j)&:=b_j(\bar s_j)-b_{j-1}(\bar s_{j-1})\,
		\text{ and }\nonumber\\
		\tilde b^{(n)}_j(\bar s^{(n)}_j)&:=b^{(n)}_j(\bar s^{(n)}_j)-b^{(n)}_{j-1}(\bar s^{(n)}_{j-1})
    \end{align}
    and
    \begin{align}\label{ddif}
        \tilde d_j(\bar s_j)&:=d_j(\bar s_j)-d_{j-1}(\bar s_{j-1})\,
		\text{ and }\nonumber\\
		\tilde d^{(n)}_j(\bar s^{(n)}_j)&:=d^{(n)}_j(\bar s^{(n)}_j)-d^{(n)}_{j-1}(\bar s^{(n)}_{j-1}),
    \end{align}
	with $b_0(\emptyset)\equiv b^{(n)}_0(\emptyset)\equiv0$,
	and~\eqref{eqcortolema2}
	follows from an application of the lemma, bearing in mind the assumption made on $\ve'$ at the beginning of this proof.

	We may then apply Lemma~\ref{LemaPre2} to $\cI_{j}(\bar s_{j-1})$ and $\cI^{(n)}_{j}(\bar s_{j-1}^{(n)})$,
	to find a common $\ell_j\geq\ell$ for 
	which the first two items of~\eqref{eqcortolema}, but for the last inequality, as well as~\eqref{eqcortolema2} 
	hold with probability larger than $1-\d2^{-j-1}/\check N_{j-1}$.
	
	Since again 
	$N_j(\bar s_{j-1})$ is finite for each corresponding $\bar s_{j-1}$, we may find a common $\hat N_j$ for all intervals such that 
	$N_j(\bar s_{j-1})\leq \hat N_j$ with probability larger than $1-\d2^{-j-1}/\check N_{j-1}$. It follows that~\eqref{eqcortolema} holds
	up to level $j$ with $1-(1-2^{-j})\d$, and the induction step is complete. Going to level $k$ establishes the result.
\end{proof}

\begin{observacao}\label{subs}
	It follows from Corollaries~\ref{Corotolema2} 
	that, given $\La,\ell>0$, there exist a subsequence of $(n')$ and  (random) $\ell_1\geq\ell,\ldots,\ell_k\geq\ell$ such that  
	for all $\bar s^{(n')}_k\in\mt^{(n')}_k,\,\bar s_k\in\mt_k$,
	$$\cQ^{(n')}_{j}(\bar s^{(n')}_{j-1})=\cQ_{j}(\bar s_{j-1})=\tcq^{(n')}_{j}(\bar t^{(n')}_{j-1}),$$
	$\max_{\bar s_k\in\mt_k}\max_{1\leq j\leq k} \Big\{\,\big|b^{(n')}_j(\bar{s}^{(n')}_j)-b_j(\bar{s}_j)\big|\,\,\bigvee\,\,\big|d^{(n')}_j(\bar{s}^{(n')}_j)-d_j(\bar{s}_j)\big|\,\Big\}\to0$
	as $n'\to\infty$.
\end{observacao}

\begin{observacao}\label{incr}
	In the previous remark, and the discussion and results leading to it, we may and will take $\ell_j=\ell_j(\ell)$ 
	non-decreasing in $\ell$, $1\leq j\leq k$.
\end{observacao}


\subsubsection{Proof of Theorem~\ref{PrinciTeor}}

We recall that this is a weak convergence result. Our strategy is to consider the coupled versions of $\cz^{(n)}$ and $\cz$ described in Subsubsection~\ref{SecCoupling}, and proceed via a standard subsubsequence argument. Given a subsequence of $(n)$ and $\La$, 
we claim to be able to use Remarks~\ref{subs} and~\ref{incr}  to find a subsubsequence $(n'')$ and a {\em time distortion} $\zeta=\zeta^{(n'')}$ such that almost surely
\begin{equation}\label{jconv1}
	\max_{1\leq j\leq k}\sup_{0\leq t\leq L}|Z^{(n'')}_j(t)-Z_j(\zeta(t))|\to0\,\text{ as }n''\to\infty.
\end{equation}
and 
\begin{equation}\label{jconv2}
	\sup_{0\leq t\leq L}|\zeta(t)-t|\to0\,\text{ as }n'\to\infty,
\end{equation}
where $L=\La^{-1}E$, with $E$ a standard exponential random variable, independent of the dynamics, as before. Upon resorting to Proposition 5.3
in Chapter 3 of~\cite{EK}, and a standard subsubsequence argument, as anticipated above, the result follows

It remains to argue the claim. Let $l_1,\,i=1,2,\ldots$ be an increasing, divergent sequence of integers.
Given $\ell=l_1$, we take the subsequence 
$(n')$ prescribed in Remark~\ref{subs}, and consider the following {\em time distortion} $\zeta_1=\zeta^{(n)}_1$ \footnote{dropping the prime superscript on $n$} such that 
$\zeta_1:[0,\infty)\to[0,\infty)$, $\zeta_1(0)=0$, $\zeta_1(t)=t$, if $t\geq L$, and
\begin{itemize}
	\item $\zeta_1$ is linear outside $\bigcup_{j=1}^k\bigcup_{\bar s_j\in\mt_j}\{b_j(\bar s_j),d_j(\bar s_j)\}$, and
	\item  $\zeta_1$ maps $b_j(\bar s_j)$ to $b^{(n)}_j(\bar s^{(n)}_j)$ and $d_j(\bar s_j)$ to $d^{(n)}_j(\bar s^{(n)}_j)$,
	$\bar s_j\in\mt_j$, $\bar s^{(n)}_j\in\mt^{(n)}_j$, $1\leq j\leq k$ .
\end{itemize}

\begin{observacao}\label{distor}
	\begin{enumerate}
		\item $\zeta_1$ maps each $I_j(\bar s_j)$ onto $I^{(n)}_j(\bar s_j)$, 
		$\bar s_j\in\mt_j$, $\bar s^{(n)}_j\in\mt^{(n)}_j$, $1\leq j\leq k$;
		\item In the discrete case, we may have $\ell_1,\ldots,\ell_k$ such that $\cQ_j(\bar s_{j-1})=\cR_j(\bar s_{j-1})$, 
		$\bar s_j\in\mt_j$, $1\leq j\leq k$;
		and, given the gap condition,
		$n_0$ may be chosen such that $\cQ^{(n)}_j(\bar s^{(n)}_{j-1})=\cR^{(n)}_j(\bar s^{(n)}_{j-1})$,
		for all $\bar s^{(n)}_j\in\mt^{(n)}_j$, $1\leq j\leq k$, and $n$.
		\item It follows from the previous item that, with those choices of $\ell_1,\ldots,\ell_k$ and $n_0$, we have that
		$$
		\bigcup_{\bar s_j\in\mt_j}I_j(\bar s_j)=\bigcup_{\bar s^{(n)}_j\in\mt^{(n)}_j}I^{(n)}_j(\bar s^{(n)}_j)=[0,L), \,1\leq j\leq k,\,n\geq n_0.
		$$
	\end{enumerate}
\end{observacao}

The claim follows readily in the discrete case, with the choice of $(n'')=(n')$ and $\zeta=\zeta_1$,  from Remark~\ref{distor} since the differences in~\eqref{jconv1} vanish as soon as $n'$ is large enough, and~\eqref{jconv2}  follows from Remark~\ref{subs}.

It remains to argue the dense case. Let us write $(n')$ used in the definition of $\zeta_1$ above as $(n_{1i})_{i\geq1}$. 
By Remarks~\ref{subs} and~\ref{incr}  applied to $\ell=l_2$, we may find a sub subsequence of $(n_{1i})_{i\geq1}$, denoted $(n_{2i})_{i\geq1}$ 
and an associated sequence of time distortions $\zeta_2=\zeta_2^{(n_{2i})}$, $i\geq1$, defined as $\zeta_1$ except that $\mt_\cdot^\cdot$ refer to $\ell=l_2$ rather than $l_1$, and accordingly for $b_\cdot^\cdot(\cdot)$ and $d_\cdot^\cdot(\cdot)$. 

We proceed inductively. Given $m\geq 2$ and $(n_{mi})_{i\geq1}$, we find a subsubsequence $(n_{m+1,i})_{i\geq1}$ prescribed in 
Remarks~\ref{subs} and~\ref{incr} applied to $\ell=l_{m+1}$ and define a 
sequence of time distortions $\zeta_{m+1}=\zeta_{m+1}^{(n_{m+1,i})}$, $i\geq1$, similarly as $\zeta_1$, but with $\mt_\cdot^\cdot$, $b_\cdot^\cdot(\cdot)$ and $d_\cdot^\cdot(\cdot)$ associated to $l_{m+1}$.  It follows from Remark~\ref{subs} that for every $m\geq1$
\begin{equation}\label{jconv3}
	\sup_{0\leq t\leq L}|\zeta_m(t)-t|\to0\,\text{ as }i\to\infty,
\end{equation}
and, recalling the first item of Remark~\ref{distor}, $|Z^{(n_{mi})}_j(t)-Z_j(\zeta_m(t))|\to0$ as $i\to\infty$ for 
$t\in\bigcup_{\bar s_j\in\mt_j}I_j(\bar s_j)$, $1\leq j\leq k$, where $\mt_\cdot$ refers to the set associated to $l_m$.
It also follows, also from the latter remark, that 
\begin{equation}\label{jconv4}
	\limsup_{i\to\infty}\max_{1\leq j\leq k}\sup_{0\leq t\leq L}|Z^{(n_{mi})}_j(t)-Z_j(\zeta_m(t))|\leq a_{l_m}.
\end{equation}

We finally, resort to a diagonal subsequence argument, by considering $(n''_i)_{i\geq1}$ such that $n''_i=n_{ii}$ and 
$\zeta=\zeta^{(n''_i)}=\zeta_i^{(n''_i)}$, $i\geq1$, and the claim follows readily from~\eqref{jconv3} and~\eqref{jconv4}
for this choice.

\begin{observacao}\label{rep}
	At this point we may draw attention an advantage of the choice of representation of the Cascading Jump Evolution as $\cz_k$ ---see~\eqref{Z}---, rather then as $\cX_k$ --- see beginning of Section~\ref{cje}. The inequality in~\eqref{jconv4} does not hold in general for the latter representation: (in the dense case) nothing prevents a large value of $\ga_i(\cdot|\bar s_{i-1})$ outside $\bigcup_{\bar s_j\in\mt_j}I_j(\bar s_j)$ if $i<j$, and similarly for the approximating versions of those quantities. 
	
	Conceivably, a version of Theorem~\ref{PrinciTeor} holds with $\cX_k$ replacing $\cz_k$, but further restrictions would be needed on the underlying CFJF's.
	See Remark~\ref{nrep} below. In any case, this would in some sense be unnecessary, since $\cX_k$ are $\cz_k$ are in one to one correspondence, and, as pointed out 
	at the introduction, $\cz_k$ is a natural representation, at least in the context of the aging.
\end{observacao}


\section{Convergence of the Bouchaud and Dean Trap Model}
\label{cbdtm}

In this section we apply Theorem~\ref{PrinciTeor} of the previous section to obtain Theorems~\ref{teo:erg},~\ref{teo:pserg}, and~\ref{teo:1serg}, as anticipated at the introduction. We devote a subsection for each argument.

For  next two subsections, we describe the unscaled BDTM as CJE induced by a CFJF as follows.
Going back to Subsection~\ref{mod}, we have the first level component of our prospective CFJF given by
\begin{equation}\label{tl1}
	S_1(r)=\sum_{x_1=1}^{M_1}\tau^1_{x_1}N^1_{x_1}(r).
\end{equation}
Notice that $\ct_1=\bigcup_{x_1=1}^{M_1}\P^1_{x_1}$, where $\P^1_{x_1}$, $x_1\geq1$, are the Poisson point processes associated to the
Poisson (counting) processes $N^1_{x_1}$, $x_1\geq1$, respectively. In order to conform to the notations of both the introduction and Section~\ref{contcas}, let us keep track of the labels in $\T_1$ by introducing $\eta^1:\T_1\to\{1,2,\ldots\}$ such that 
$\eta^1(s_1)=x_1$ if the (Poisson) point $s_1$ belongs to $\P^1_{x_1}$.

Proceeding inductively, given $S_i(\cdot|\bar s_{i-1})$, $\eta^i(\cdot|\bar s_{i-1})$, $\bar s_{i}\in\ct_{i}$, $1\leq i<k$, where $\eta^1(\cdot|\emptyset)=\eta^1$, let
\begin{equation}\label{tl2}
	S_{i+1}(r|\bar s_i)=\sum_{x_{i+1}=1}^{M_{i+1}}\tau^{i+1}_{\bar x_{i+1}}N^{i+1}_{x_{i+1}}(r|\bar s_i),
\end{equation}
where $x_j=\eta^j(\cdot|\bar s_{j-1})$, $j\leq i$; 
$N^j_{x_j}(\cdot|\bar s_{j-1})$, $x_j\geq1$, $\bar s_{j-1}\in\ct_{j-1}$, $j=1,\dots, k$,  are independent standard Poisson processes,
with $N^1_{x_1}(\cdot|\emptyset)=N^1_{x_1}$; and for $\bar s_i\in\ct_i$ let $\eta^{i+1}(\cdot|\bar s_i):\T_{i+1}(\bar s_i)\to\{1,2,\ldots\}$ such that 
$\eta^{i+1}(s_{i+1}|\bar s_i)=x_{i+1}$ if $s_{i+1}$ belongs to the point process associated to $N^{i+1}_{x_{i+1}}(\cdot|\bar s_i)$.

$\cz_k$ is then the CJE associated to the $k$-CFJF $\ss_k=\{S_j(\cdot| \bar s_{j-1});\,\bar s_{j-1}\in\ct_{j-1},\,1\leq j\leq k\}$ as defined above.
In this (finite volume) context, we may write $\ct_j=\ct_j^M$, $1\leq j\leq k$.

In the next subsections we assume the particular distribution for the (random) family $\tau_\cdot^\cdot$ prescribed at the introduction and in the paragraph of~\eqref{tau}. In the next two subsections, we impose our fine tuning relationship among the volumes prescribed  in the paragraph right below that of~\eqref{tau}, and establish Theorems~\ref{teo:erg},~\ref{teo:pserg} by writing the rescaled corresponding processes as CJE's induced by suitable corresponding rescaled CFJF's. Upon subsequently obtaining the scaling limits of those  CFJF's, and appropriately invoking Theorem~\ref{PrinciTeor}, Theorems~\ref{teo:erg},~\ref{teo:pserg} follow.

\subsection{Proof of Theorem~\ref{teo:erg}}\label{proof1}

The ergodic scaling of Theorem~\ref{teo:erg} corresponds to the rescaling of $\tau_\cdot^\cdot$ as follows: 
\begin{equation}\label{rtau1}
	\ga_\cdot^{(n)}(\cdot|\cdot)=\frac1{n^{1/\a_1}}\tau_\cdot^\cdot,
\end{equation}
where the correspondence between elements of $\ct_\cdot$ and labels of  $\tau_\cdot^\cdot$ are as explained in the paragraphs of~\eqref{tl1}
and~\eqref{tl2}. We thus make
\begin{equation}\label{rs1}
	S_j^{(n)}(r|\bar s_{j-1})=\frac1{n^{1/\a_1}}S_j(r|\bar s_{j-1})=\frac1{n^{1/\a_1}}\sum_{x_{j}=1}^{n^{\a_j/\a_1}}\tau^{j}_{\bar x_{j}}N^{j}_{x_{j}}(r|\bar s_{j-1}),
\end{equation}
for $\bar s_{j-1}\in\ct^{(n)}_{j-1}$, $1\leq j\leq k$,
and it follows from a classical limit theorem for iid random variables attracted to stable laws of indices in $(0,1)$ that
\begin{equation}\label{rs2}
	S_j^{(n)}(r|\bar s_{j-1})\to \tilde S_j(r|\bar s_{j-1}):=\sum_{x\in[0,1]}\ga_j(x|\bar s_{j-1})\tilde N^{j}_{x}(r|\bar s_{j-1})
\end{equation}
in distribution as $n\to\infty$, where $\ga_j(\cdot|\bar s_{j-1})$, $\bar s_{j-1}\in\ct_{j-1}$, are the increment of
independent $\a_j$-stable subordinators,  $1\leq j\leq k$,  which are also independent when $j$ varies.
And $\tilde N^{j}_{x}(\cdot|\bar s_{j-1})$, $x\in[0,1]$, $\bar s_{j-1}\in\ct_{j-1}$, $1\leq j\leq k$, are independent standard Poisson processes.
We may identify the Poisson processes entering both sides of~\eqref{rs2} by taking an enumeration of the countable collection of $x$'s in the sum on the 
right hand side for which $\ga_j(x|\bar s_{j-1})>0$, in the natural way, and, noticing that $\ct^{(n)}_j$ may be embedded (also in the natural way)
in $\ct_j:=\bigcup_{n\geq1}\ct^{(n)}_j$, make sense of the use of $\bar s_{j-1}\in\ct_j$ on both sides.

In order to apply Theorem~\ref{PrinciTeor} we need the convergence in~\eqref{rs2} to be strong, however, so we resort to Skorohod Representation to
go to another probability where we have almost sure convergence in~\eqref{rs2}, and now we may apply Theorem~\ref{PrinciTeor} to readily obtain Theorem~\ref{teo:erg}.

\subsection{Proof of Theorem~\ref{teo:pserg}}\label{proof2}

The scaling now is given by
\begin{equation}\label{rtau2}
	\ga_\cdot^{(n)}(\cdot|\cdot)=\frac1{n^\b}\tau_\cdot^\cdot,
\end{equation}
and the suitable form for the rescaled timeless clocks is as follows
\begin{align}\label{rs3}
	S_j^{(n)}(r|\bar s_{j-1})&=\frac1{n^{\b}}S_j^{(n)}(n^{-\chi_j}r|\bar s_{j-1})\nonumber\\
	&=\frac1{n^{\b}}\sum_{x_{j}=1}^{n^{\a_j/\a_1}}\tau^{j}_{\bar x_{j}}N^{j}_{x_{j}}\big(n^{-\chi_j}r|\bar s_{j-1}\big),
\end{align}
$\bar s_{j-1}\in\ct^{(n)}_{j-1}$, $1\leq j\leq k$,
where $\chi_j=\a_j(1/\a_1-\b)$. Recall the fourth item of Remark~\ref{X-process}.

We now claim that for almost every $\tau_\cdot^\cdot$
\begin{equation}\label{rs4}
	S_1^{(n)}\to \hat S_1
\end{equation}
in distribution, where $\hat S_1$ is an $\a_1$-stable subordinator. Resorting to Skorohod Representation, we have~\eqref{rs4} as an almost sure convergence. Proceeding inductively, we obtain a CFJF $\hat\ss_k=\{\hat S_j(\cdot| \bar s_{j-1});\,\bar s_{j-1}\in\ct_{j-1},\,1\leq j\leq k\}$ consisting of independent
subordinators, such that $\hat S_j(\cdot|\cdot)$ has index $\a_j$, $2\leq j\leq k$, and for each $\bar s_{j-1}\in\ct_{j-1}$, $1\leq j\leq k$, we may produce a compounded matching sequence $\bar t^{(n)}_{j-1}\in\ct^{(n)}_{j-1}$ satisfying the first bullet item of the paragraph of~\eqref{3bul} such that
for almost every $\tau_\cdot^\cdot$
\begin{equation}\label{rs5}
	S_j^{(n)}(\cdot|\bar t^{(n)}_{j-1})\to \hat S_j(\cdot|\bar s_{j-1})
\end{equation}
in distribution, and again resorting to Skorohod Representation, we have~\eqref{rs5} as an almost sure convergence.
An application of Theorem~\ref{PrinciTeor} readily yields Theorem~\ref{teo:pserg}.

In order to argue the claim, it is enough to establish~\eqref{rs4} and~\eqref{rs5}, remarking that the left hand sides
are independent when $j$ and $\bar t^{(n)}_{j-1}$ vary (for all fixed $n$), and their distribution do not depend on $\bar t^{(n)}_{j-1}$.
Those convergences may be essentially obtained from Theorem 1.4 of~\cite{GG15}, where ordinary clock processes are considered, 
instead of our timeless clock processes, but the difference is minor. They may also be formulated as the following result, whose proof 
we present in Appendix~\ref{app2} for completeness.

\begin{lema}\label{lpserg}
	Let $\tau,\tau_i$, $i\geq1$, be iid random variables such that
	\begin{equation}\label{taubis}
		\pr(\tau>t)=\frac{L(t)}{t^{\a}},\,t\geq0,
	\end{equation}
	where $L(t)\to1$ as $t\to\infty$ and $\a\in(0,1)$, and make
	\begin{equation}\label{rs6}
		S^{(n)}(r)=\frac1{n^{\nu/\a}}\sum_{i=1}^{n}\tau_iN_i\,\big(n^{-(1-\nu)}r\big),\,r\geq0, 
	\end{equation}
	where $N_i$, $i\geq1$,are iid standard Poisson processes, independent of the $\tau_i$'s, and $\nu\in(0,1)$. Then for almost every realization of $\tau_i$, $i\geq1$,
	$S^{(n)}$ converges in distribution to an $\a$-stable subordinator.
\end{lema}

\subsection{Proof of Theorem~\ref{teo:1serg}}\label{proof3}

We first take the limit as $M_1,\ldots,M_k\to\infty$ of the unscaled process. For simplicity, we take $M_j=M_j^{(n)}\to\infty$ as $n\to\infty$, $1\leq j\leq k$,
in any arbitrary way. Other ways of taking the limit can be similarly treated. Since this first limit is supposed to involve a discrete CFJF, it is convenient to
write the jump functions as
\begin{equation}\label{rs11}
	S_j^{(n)}(r|\bar s_{j-1})= \sum_{\ell=1}^{\lfloor r\rfloor}\tau^{j}_{U^j_{\ell}(\bar s_{j-1})\bar x_{j-1}},        
\end{equation}
where 
$\{U^j_\cdot(\cdot),\,1\leq j\leq k\}$, are independent families of iid random variables,
uniformly distributed in
$\{1,\ldots,M_j\}$, respectively, and  $x_i=U^i_{s_i}(\bar s_{i-1})$, $\bar s_{i-1}\in\ct^{(n)}_{i-1}\equiv\{1,2\,\ldots\}^{i-1}$,  $1\leq i<j$. 

We claim that for almost every $\tau_\cdot^\cdot$
\begin{equation}\label{rs12}
	S_j^{(n)}(\cdot|\bar s_{j-1})\to\check S_j(\cdot|\bar s_{j-1}):=\sum_{\ell=1}^{\lfloor \cdot\rfloor}\tau^{j}_{\ell\bar s_{j-1}}, 
\end{equation}
in distribution, for all $\bar  s_{j-1}\in\ct_{j-1}$. 

We then proceed, along the lines already traced in the above subsections, resorting to Skorohod Representation,
to conclude that for almost every $\tau_\cdot^\cdot$, $\cz$ converges in distribution as $n\to\infty$ to $\check\cz$, the CJE induced by
$\check\ss_k=\{\check S_j(\cdot| \bar s_{j-1});\,\bar s_{j-1}\in\ct_{j-1},\,1\leq j\leq k\}$.

The claim is readily implied by the following result, whose proof we postpone to the end of the subsection.

\begin{lema}\label{l1serg}
	Let $X,X_i$, $i\geq1$, be iid random variables, and $U_i$, $i\geq1$, iid random variables which are uniformly distributed in $\{1,\ldots,M\}$, and independent of the $X_i$'s. Then, for almost every $X_\cdot$
	\begin{equation}\label{rs13}
		(X_{U_1},X_{U_2},\ldots)\to(X_{1},X_{2},\ldots)
	\end{equation}
	in distribution  (in the product topology) as $M\to\infty$.
\end{lema}

We note that $\check\ss_k$ is discrete, and that the convergence in~\eqref{rs12} is readily seen to satisfy the gap condition.

We thus have that for almost every $\tau_\cdot^\cdot$
$\check\cz^{(m)}\to\frac1m\check\cz(m\cdot)$
in distribution as $M_1,\ldots,M_k\to\infty$, for every $m$. 

In order to take the limit in $m$, and thus establish Theorem~\ref{teo:1serg}, we simply observe that the latter process is a CJE induced by $\ss_k^{(m)}$ such that
\begin{equation}\label{rs14}
	S_j^{(m)}(\cdot|\bar s_{j-1})=\frac1m\sum_{x_j=1}^{\lfloor  m^{\a_j}\cdot\rfloor}\tau^{j}_{\bar x_{j}}, 
\end{equation}
$\bar s_{j-1}\in\ct^{(m)}_{j-1},\,1\leq j\leq k$, where the correspondence between $\bar s_{j-1}\in\ct^{(m)}_{j-1}$ and $\bar x_{j-1}$ is (the obvious one) such
that $s_i$ is the $x_i$-th jump of $\T^{(m)}_i(\bar s_{i-1})$ in order (notice that the latter set is discrete).

The result then follows from the classical result saying that the right hand side of~\eqref{rs14} converges (in distribution) to an $\a_j$-stable subordinator,
by applying Theorem~\ref{PrinciTeor}, after resorting to Skorohod Representation, as before.

\paragraph{Proof of Lemma \ref{l1serg}.} It is enough to show that, given $\ell\geq1$, and Borel sets $B_1,\ldots,B_{2\ell}$, we have that
\begin{align}\label{rs15}
	\pr(X_{U_1}\in B_1,\ldots,&X_{U_\ell}\in B_\ell,X_{1}\in B_{\ell +1},\ldots,X_{\ell}\in B_{2\ell})\nonumber\\
    &\to \prod_{i=1}^{2\ell}\pr(X\in B_i),
\end{align}
as $M\to\infty$. In order to do that, consider the following events
\begin{equation}\label{rs16}
	C=\bigcap_{i=1}^\ell\{U_i>\ell\},\quad D=\bigcap_{i,j=1\atop i\ne j}^\ell \{U_i\ne U_j\}.
\end{equation}
It is quite clear that $\pr(C\cap D)\to1$ as $M\to\infty$. The claim of the lemma then follows from this and the fact, that may be readily checked, that
$$\pr(A\cap C\cap D)= \pr(C\cap D) \prod_{i=1}^{2\ell}\pr(X\in B_i),$$
where $A$ is the event inside the probability sign on the left hand side of~\eqref{rs15}. 

\begin{observacao}\label{nrep}
	Let us come back to Remark~\ref{rep} in the context of the BDTM. As pointed out in that remark, we would need further conditions to have the convergence of $\cX_k^{(n)}$. It can be worked out that we would need to further require monotonicity of the $\a_j$'s in the levels: $\a_1<\cdots<\a_k$. Without this, we 
	would see large values of $\ga_i(\cdot|\bar s_{i-1})$ outside $\bigcup_{\bar s_j\in\mt_j}I_j(\bar s_j)$ for $i<j$, and this would prevent the convergence in the $J_1$ Skorohod topology. 
	
	Curiously ---perhaps intriguingly---, this representation and monotonicity condition are both adopted in~\cite{BD}, but nothing seems to indicate that the condition is used in the derivation of the (aging) results therein (as pointed out at the introduction of the present paper). Indeed it is not necessary, as is perhaps clear from the results in this section (from the one to one correspondence of $\cX^{(n)}$ and $\cz^{(n)}$). We come back to this point in more detail in Appendix~\ref{app3}.
	
\end{observacao}

\begin{observacao}\label{nft}
	It is perhaps quite clear from the arguments in this section that in order to obtain convergence results for $\cz$, it is enough to obtain corresponding suitable convergence results for the jump functions in each level of the underlying tree, separately. For this reason, going beyond our fine tuning assumption on the volumes $M_1,\ldots, M_k$ does not pose extra difficulties. For example, we might make a choice leading to ergodic scalings in some levels and polynomial aging scalings in other levels. This would complicate the description of the results, with no essentially greater complexity, so we leave the matter at that. 
	
	(It should perhaps be said that going to longer time scales than the ergodic time scale would yield irregular trajectories for the limiting process, namely, independent observations of the equilibrium distribution at different times, thus disabling convergence in the $J_1$ topology. But this is a straightforward issue.)
\end{observacao}

\begin{acknowledgments}
We would like to thank Renato dos Santos for raising the discontinuity issue discussed in 
Remark~\ref{rmk:j1_cont}. Thanks also to an anonymous referee, whose attentive and thoughtful reading helped correct and also otherwise improve the text.
\end{acknowledgments}

\appendix

\section{Proof of Lemma \ref{LemmaExercicio} }\label{app1}

\setcounter{equation}{0}

\begin{proof}[\textbf{Item (a)}]
	Consider the event $A=\{K(R-)<L<K(R)\}$.
	To show that $R\in\T_S$, it is enough to argue the claim that $\pr(A)=1$, equivalently $\pr(A^c)=0$.  
	Observe that $A^c=\{L\in\fR\}$, where 
	and $\fR$ is the range of $K(\cdot)$.  Note that $L$ is an exponential random variable, independent of $\{\E_s,s\in\T_S\}$, so $L$ is independent of $\fR$, and thus 
	\begin{equation}
		\pr(L\in\fR)=\Er(P(L\in\fR|\fR))=\Er\left(\int_{\fR}\Lambda e^{-\Lambda x}dx\right)=0,
	\end{equation}
	since $\fR$ is almost surely a set of  Lebesgue measure zero, and the claim is established. 
	
	\smallskip
	
	
	\noindent {\it \textbf{Items (b-c)}} 
	Let us make the following computation. 
	%
	Given $k\geq 1$, $r\in\T_S$, and $s_1<\cdots<s_k<r$,
	\begin{align}
		&\pr(R=r,\E_{s_1}>t_1,\ldots,\E_{s_k}>t_k,\hat{\E}_r>t)\nonumber\\
		&=\int{f_{W}(w)dw}\int_{t_1}^{\infty}{f_{\E_{s_1}}(w_1)dw_1}\nonumber\\
        &\hspace{0.5cm}\cdots
		\int_{t_k}^{\infty}{f_{\E_{s_k}}(w_k)dw_k}\,\pr(\E_r>L-\bar{w}_k-w;L>\bar{w}_k+w+t)\nonumber\\
		&=\int{f_{W}(w)dw}\int_{t_1}^{\infty}{f_{\E_{s_1}}(w_1)dw_1}\nonumber\\
        &\hspace{0.5cm}\cdots
		\int_{t_k}^{\infty}{f_{\E_{s_k}}(w_k)dw_k}\int_{\bar{w}_k+w+t}^{\infty}{\Lambda e^{-\Lambda u}e^{-\lambda_r(u-\bar{w}_k-w)}du}\nonumber\\
		&=\frac{\Lambda}{\Lambda+\lambda_r}e^{-(\lambda_r+\Lambda)t}\int{f_{W}(w)dw}\int_{t_1}^{\infty}{f_{\E_{s_1}}(w_1)dw_1}\nonumber\\
        &\hspace{0.5cm}\cdots
		\int_{t_k}^{\infty}{f_{\E_{s_k}}(w_k)dw_k}e^{-\Lambda(\bar{w}_k+w)}\nonumber\\
		&=\frac{\Lambda}{\Lambda+\lambda_r}e^{-(\lambda_r+\Lambda)t}\int{e^{-\Lambda w}f_{W}(w)dw}\int_{t_1}^{\infty}{e^{-\Lambda w_1}f_{\E_{s_1}}(w_1)dw_1}\nonumber\\
        &\hspace{0.5cm}\cdots\int_{t_1}^{\infty}{e^{-\Lambda w_k}f_{\E_{s_k}}(w_k)dw_k}\nonumber\\
		&=\prod_{s<r}\frac{\lambda_s}{\lambda_s+\Lambda}\frac{\Lambda}{\lambda_r+\Lambda}\prod_{i=1}^{k}e^{-(\lambda_{s_i}+\Lambda)t_i}e^{-(\lambda_r+\Lambda)t}\nonumber\\
        &=\pr(R=r)\prod_{i=1}^k\pr(\tilde{\E}_{s_i}>t_i)\pr(\tilde{\E}_r>t),
	\end{align}
	where 
	$W=\displaystyle\sum_{\substack{s<r:\,  s\neq s_1,\ldots,s_k }}\E_s$.
	\begin{equation}
		(\E_{s_1},\ldots,\E_{s_k},\hat{\E}_r)|R=r\sim (\tilde{\E}_{s_1},\ldots,\tilde{\E}_{s_k},\tilde{\E}_r),
	\end{equation}
	as well as~\eqref{R}, follow, and this is enough.

	\medskip
	
	{\it \textbf{Item (d)}} 
	\begin{align*}
		&\pr(\hat{\E}_{r}>x,\check{\E}_r>y,R=r)\\
        &=\pr(L-K(r-)>x,K(r)-L>y)\\
		&=\pr(L>Z+x,\E_r>L+y-Z)\\
		&=\int{f_Z(z)dz}\int_{z+x}^{\infty}{f_L(u)du}\int_{u+y-z}^{\infty}{f_{\E_r}(v)dv}\\
		&=\frac{\Lambda}{\lambda_r+\Lambda}e^{-\lambda_ry}e^{-(\lambda_r+\Lambda)x}\int{f_Z(z)dz}\\
		&=\frac{\Lambda}{\lambda_r+\Lambda}e^{-\lambda_ry}e^{-(\lambda_r+\Lambda)x}E[e^{-\Lambda Z}]\\
		&=e^{-\lambda_ry}e^{-(\lambda_r+\Lambda)x}\prod_{s<r}\frac{\lambda_s}{\lambda_s+\Lambda}\frac{\Lambda}{\lambda_r+\Lambda}\\
		&=\pr(\check{\E}_r>y)\pr(\hat{\E}_{r}>x)\pr(R=r),
	\end{align*}
	where $Z=\displaystyle\sum_{s<r}\E_s$, and the result follows.
\end{proof}

\section{Proof of Lemma~\ref{lpserg}  }\label{app2}

\setcounter{equation}{0}

Let $\tau_i$, $i\geq1$, be fixed.
Since $S^{(n)}$ is itself a subordinator, it is enough, by Corollary 3.6 in Chapter 7 of~\cite{J13}, to show the convergence of the Laplace exponent
of $S^{(n)}$, namely
\begin{equation}\label{lapex}
	\varphi^{(n)}(\theta)=\frac1{n^{1-\nu}}\sum_{i=1}^n \big(1-e^{-\theta\tau_i/n^{\nu/\a}}\big),\,\theta>0,
\end{equation}
as $n\to\infty$ almost surely to a constant times $\theta^\a$. We write the right hand side of~\eqref{lapex} as
\begin{equation}\label{le1}
	\frac1{n^{1-\nu}}\sum_{j=1}^{n^{1-\nu}}W_j^{(n)}:=
	\frac1{n^{1-\nu}}\sum_{j=1}^{n^{1-\nu}}\sum_{i\in B_j^{(n)}} \big(1-e^{-\theta\tau_i/n^{\nu/\a}}\big),
\end{equation}
where $B_j^{(n)}$, $1\leq j\leq n^{1-\nu}$ is a partition of $\{1,\ldots,n\}$ in intervals of (roughly) equal size $n^\nu$.
The summands the of the outer sum in~\eqref{le1} are independent and have essentially the same distribution. Here and below, details are somewhat sketchy, disregarding the proper use of integer parts of $n^\nu$ and $n^{1-\nu}$.
We trust the main steps to be clear, and the necessary adjustments to account for the latter points to be straightforward
(if cumbersome).

The result follows from the following claims.
\begin{equation}\label{le2}
	\frac1{n^{1-\nu}}\sum_{j=1}^{n^{1-\nu}}\bar W_j^{(n)}\to0
\end{equation}
almost surely as $n\to\infty$, where $\bar W_j^{(n)}=W_j^{(n)}-\Er(W_j^{(n)})$;
\begin{equation}\label{le3}
	\frac1{n^{1-\nu}}\sum_{j=1}^{n^{1-\nu}}\Er(W_j^{(n)})\to\text{ const } \theta^\a
\end{equation}
as $n\to\infty$.

Both claims may be seen to follow from known results relying on the following bound.
\begin{equation}\label{le4}
	\sup_{n}\Er(e^{W_1^{(n)}})<\infty
\end{equation}

For~\eqref{le2} we resort to Corollary 2 of~\cite{HSV}, the conditions for which are all readily seen to apply, except perhaps for the existence of a random variable $X$ with a suitable polynomial moment (a second moment in the present case) such that
$\sup_{n}\pr\big(W_1^{(n)}>x\big)\leq D\pr(X>x)$ for all $x$ and some constant $D$.~\eqref{le4} and Markov's inequality readily implies this condition with $X$ exponentially distributed with mean 1 and $D$ as the left hand side of~\eqref{le4}.

As for~\eqref{le3}, it follows from well known results that
\begin{equation}\label{le5}
	W_1^{(n)}\to W=\sum_{x\in[0,1]}\big(1-e^{-\theta\ga_x}\big)
\end{equation}
in distribution as $n\to\infty$, where $\ga_x$ are the increments of an $\a$-stable subordinator. Now~\eqref{le4} implies that
$W_1^{(n)}$, $n\geq1$, are uniformly integrable, and thus 
\begin{equation}\label{le6}
	\Er\big(W_1^{(n)}\big)\to \Er(W)
\end{equation}
as $n\to\infty$; Campbell's Theorem then implies that $\Er(W)=$ const $\theta^\a$. A direct computation involving an integral like the one in~\eqref{le9} below also yields this result.

It remains to check~\eqref{le4}. It is enough to show that
\begin{equation}\label{lex}
	\limsup_{n}\Er(e^{W_1^{(n)}})<\infty.
\end{equation}

The expectation may be written as 
\begin{equation}\label{le7}
	\big\{\Er\big(e^{1-e^{-\theta\tau/m^{1/\a}}}\big)\big\}^m,
\end{equation}
where $m=n^\nu$. 
Setting $\theta_m=\theta/m^{1/\a}$, we write the expected value in~\eqref{le7} as $1+\Psi_m(\theta)$, where
\begin{align}\nn 
	\Psi_m(\theta)&=\int_0^\infty (e^{1-e^{-\theta_m x}}-1)\,dF(x)\\
    &\stackrel{\text{\tiny by parts}}
	=\int_0^\infty\bar F(x) (e^{1-e^{-\theta_m x}}-1)e^{-\theta_m x}\theta_m dx\nonumber \\ 
	&=\int_0^\infty\bar F(x/\theta_m) (e^{1-e^{-x}}-1)e^{-x}dx,\label{le8}
\end{align}
where $F$ is the distribution function of $\tau$ and $\bar F=1-F$. We have from~\eqref{taubis} that
$\bar F(x/\theta_m)=\frac1m\frac{\theta^\a}{x^\a}\,L(m^{1/\a}\theta^{-1}x)$. It then follows from the asymptotics of $L$
and dominated convergence that 
\begin{equation}\label{le9}
	m\Psi_m(\theta)\to\theta^\a\int_0^\infty\frac1{x^\a} (e^{1-e^{-x}}-1)e^{-x}dx	
\end{equation}
as $m\to\infty$, and~\eqref{lex} follows from this and~\eqref{le7}.

\section{Aging}\label{app3}

\setcounter{equation}{0}

We called the scalings underlying Theorems~\ref{teo:pserg} and~\ref{teo:1serg} {\em aging} scalings, and this is related to the self-similarity of the distribution of
$\hat\cz$, as follows. Let us remark that, given $\theta>0$, $\hat\cz^{\theta}(\cdot):=\theta^{-1}\hat\cz(\theta\cdot)$ is a CJE induced by 
$\hat\ss^\theta=\{\hat S^\theta_j(\cdot| \bar s_{j-1});\,\bar s_{j-1}\in\ct^\theta_{j-1},\,1\leq j\leq k\}$, where
\begin{equation}\label{sca}
	\hat S^\theta_j(\cdot| \bar s_{j-1})=\theta^{-1}\hat S_j(\cdot \,\theta^{\a_j}| \bar s_{j-1}).
\end{equation}
Recall the  fourth item of Remark~\ref{X-process}. The well known scale invariance of stable subordinators says that the distribution of the right hand side of~\eqref{sca} does not depend on $\theta$, and, thus, neither does that of $\hat\cz^{\theta}(\cdot)$. 

This may be seen as an aging property of $\hat\cz$. In particular, from Theorems~\ref{teo:pserg} and~\ref{teo:1serg}, given a(n almost surely) continuous function of the portions of the trajectory of $\cz$ in $[0,t_w]$ and  $(t_w+t]$,
where $t_w,t>0$ are fixed, $G_{t_w,t}(\cz)=\cg\big(\cz|_0^{t_w},\cz|_{t_w}^{t_w+t}\big)$, we have that for almost every $\tau_\cdot^\cdot$
\begin{align}\label{al1}
	&\lim_{n\to\infty}\Er\big(G_{t_w,t}(\hat\cz^{(n)}\big)\\
    &=\lim_{m\to\infty}\lim_{M_1,\ldots,M_k\to\infty}\Er\big(G_{t_w,t}(\check\cz^{(m)})\big)\\
	&=\Er\big(G_{t_w,t}(\hat\cz)\big)\\
    &=\Er\big(G_{1,t/t_w}(\hat\cz)\big),
\end{align}
where the latter equality follows from the self-similarity of (the distribution of) $\hat\cz$. 

The functions in~\eqref{pi} have this form, with $G_{t_w,t}$ as the indicator function of the event of no jump in $[t_w,t_w+t]$, but this is {\em not} continuous in the $J_1$-Skorohod topology\footnote{it does not depend on the first portion of the trajectory, but this is of course not a problem}. In this case, however, as pointed out in Remark~\ref{Pi},
$\Pi_j(t,t_w)= \Er(e^{-t/Z_j(t_w)})$, and we may thus apply~\eqref{al1} with $G_{t_w,t}(\cz)=e^{-t/Z_j(t_w)}$ to obtain~\eqref{pip} and~\eqref{pi1} with $f_j(\theta)= \Er(e^{-\theta/\hat Z_j(1)})$, $\theta>0$, $1\leq j\leq k$.

Another aging function commonly considered is (in the present context) 	
$$R_j(t,t_w)=\pr(Z_j(t_w)=Z_j(t_w+t]).$$
We again have an event inside the latter probability whose indicator function is not continuous in the $J_1$-Skorohod topology.
Going around this point to establish that 
\begin{align}\label{al2}
	\pr\big(&Z_j(n^\b t_w)=Z_j(n^\b(t_w+t))\big)\\
    &\hspace{1cm}\to\pr\big(\hat Z(t_w)=\hat Z_j(t_w+t)\big)=f_j(t/t_w)
\end{align}
as $n\to\infty$ for almost every $\tau_\cdot^\cdot$, $1\leq j\le k$, as soon as $\b<1/\a_1$ (and the corresponding result
under order 1 aging scaling) demands a longer argument (for the convergence; the latter equality in~\eqref{al2}  follows 
from fact that $Z_\cdot$ quite clearly never visits the same state twice), which we leave as an exercise.

An aging function which fits the description in the paragraph of~\eqref{al1} above is as follows
$$\pr\big(\max_{0\leq s\leq t_w}Z_j(s)<\max_{t_w\leq s\leq t_w+t}Z_j(s)\big),$$
which may be understood as the {\em prospects of novelty in the system}. The indicator function of the event in the latter probability may be readily checked to be almost surely continuous with respect to the distribution of $\hat\cz$, and we thus immediately get that
\begin{align*}
\pr\big(&\max_{0\leq s\leq n^\b t_w}Z_j(s)<\max_{n^\b t_w\leq s\leq n^\b(t_w+t)}Z_j(s)\big)\\
&\hspace{1cm}\to
\pr\big(\max_{0\leq s\leq t_w}\hat Z_j(s)<\max_{t_w\leq s\leq t_w+t}\hat Z_j(s)\big)=\tilde f(t/t_w)
\end{align*}
as $n\to\infty$ for almost every $\tau_\cdot^\cdot$, $1\leq j\le k$, as soon as $\b<1/\a_1$ (and the corresponding result
under order 1 aging scaling), where $\tilde f(\cdot)=\pr\big(\max_{0\leq s\leq 1}\hat Z_j(s)<\max_{1\leq s\leq 1+\cdot}\hat Z_j(s)\big)$.

In the the next subsection, we present (briefly) a derivation of a more explicit expression for $f_j$, as anticipated in the paragraph below~\eqref{pii} above.

\subsection{A more explicit expression for $f_j$}\label{fj}

It is convenient to put the `no jump beyond level $j$' event in the probability in the definition of $\Pi_j(t,t_w)$ (as in the right hand side of~\eqref{pi}) in terms of the clock processes, rather than the $\cz$ process directly (as often done in the literature, when dealing with this aging function); more specifically, the clock parts. In the case of $\hat\ss$, these processes are also stable subordinators, with the same index as the respective jump functions which originate them. Namely,
\begin{equation}\label{kj}
	\hat K_{i}(r|\cdot)=\int_0^{r}E_j(s)\,d\hat S_i(s|\cdot),\,r\geq0, 1\leq i\leq k,
\end{equation}
are $\a_i$-stable subordinators, as $\hat S_i(r|\cdot),\,r\geq0$, respectively. Let $\fR_i(\cdot)$ denote the range of 
$\{\hat K_{i}(r|\cdot),\,r\geq0\}$, and for $t>0$, let $\fy_i(t|\cdot)=\sup\{[0,t]\cap\mr_i(\cdot)\}$ and  
$\fz_i(t|\cdot)=\inf\{[t,\infty)\cap\mr_i(\cdot)\}$. It is a classical result, going back to Dynkin and Lamperti, that for every fixed $t>0$
\begin{align}\label{dl}
	&\Big(\frac1t\fy_i(t|\cdot),\frac1t\fz_i(t|\cdot)\Big)\sim(\fy_i,\fz_i),\text{ where }\nonumber\\	&\pr(\fy_i>y,\fz_i>z)=\frac{\sin\pi\a_i}\pi\int_y^1(z+u)^{-\a_i}(1-u)^{\a_i-1}du, 
\end{align}
$y\in(0,1),\,z>0$; "$\sim$" means "distributed as"; see, e.g.~Section 9 of~\cite{E}.

The 'no jump beyond level $j$' event for $\hat\cz$ may then be written as (using the self similarity of  its distribution to reduce to 
$t_w=1,t=\theta$, with an arbitrary $\theta>0$)
\begin{widetext}
\begin{equation}\label{njj}
	\big\{\fz_1(1)>\theta; \fz_2\big(\fy_1(1)|\zeta_1\big)>\theta;\ldots;\fz_j\big(\fy_{j-1}\big(\cdots\fy_2(\fy_1(1)|\zeta_1)\cdots|\bar\zeta_{j-1}\big)>\theta\big\},
\end{equation}
\end{widetext}
where $\zeta_i$ corresponds to the jump location of $Z_i$ at time 1. 
From the independence of the jump functions constituting $\hat\ss$ and~\eqref{dl} , it follows then that we may write
\begin{equation}\label{njj1}
	\Pi_j(\theta,1)=\pr\big(\fz_1>\theta;\fy_1\fz_2>\theta;\ldots;\bar\fy_{j-1}\fz_j>\theta\big),
\end{equation}
where $(\fy_1,\fz_1),\dots(\fy_k,\fz_k)$ are independent random vectors distributed as prescribed  in~\eqref{dl}, respectively.
We will now show by induction that the latter probability equals $\pr\big(\bar\fy_{j}>\theta/(1+\theta)\big)$.

Conditioning on $(\fy_1,\fz_1)$, we find that the probability in~\eqref{njj1} equals 
\begin{widetext}
\begin{equation}\label{njj2}
	\frac{\sin\pi\a_1}\pi	\int_0^1\pr\Big(\fz_2>\theta/y;\ldots;\bar\fy|_{2}^{j-1}\,\fz_j>\theta/y\Big)(\theta+y)^{-\a_1}(1-y)^{\a_1-1}dy,
\end{equation}
\end{widetext}
where $\bar\fy|_{2}^{j-1}=\prod_{i=2}^{j-1} \fy_i$. 
By the induction hypothesis, the latter probability equals $\pr\Big(\bar\fy|_{2}^{j}>\frac\theta{y+\theta}\Big)$. 
Using this, and changing the variable of integration through $x=\frac{y+\theta}{1+\theta}$, we find that the right hand side of~\eqref{njj2} equals
\begin{align}\label{njj3}
	&\frac{\sin\pi\a_1}\pi\int_{\theta'}^1\pr\big(x\,\bar\fy|_{2}^{j}>\theta'\big) x^{-\a_1}(1-x)^{\a_1-1}dy\nonumber\\
	&=\pr\big(\fy_{1}>\theta';\,\bar\fy_{j}>\theta'\big)=\pr\big(\bar\fy_{j}>\theta'\big),
\end{align}
where $\theta'=\theta/(1+\theta)$, and the induction step is complete.

\begin{observacao} 
	It is perhaps worth to emphasize, a point already brought up in~\cite{SN}, that this form of the 'no jump' correlation functions of the aging (scaling limit of the) BDTM is indeed different from the ones for the GLTM (as far as they are comparable, namely, when $\a_1<\cdots<\a_k$).
	Not only the right hand side of~\eqref{njj3} is not an arcsine law (for $j\geq2$, reflecting the fact that the product of two independent Beta($\a,1-\a$)  and Beta($\a',1-\a'$) distributed random variables does not have a Beta($\a'',1-\a''$) distribution for any $\a,\a',\a''\in(0,1)$), but their behaviors near 0 and $\infty$ also differ from the ones for the GLTM. We may safely say that these two models are in different {\em aging universality classes} (apart from other more obvious differences).
\end{observacao}
\nocite{*}
\bibliography{bibliografia}

\begin{thebibliography}{53}%
\makeatletter
\providecommand \@ifxundefined [1]{%
 \@ifx{#1\undefined}
}%
\providecommand \@ifnum [1]{%
 \ifnum #1\expandafter \@firstoftwo
 \else \expandafter \@secondoftwo
 \fi
}%
\providecommand \@ifx [1]{%
 \ifx #1\expandafter \@firstoftwo
 \else \expandafter \@secondoftwo
 \fi
}%
\providecommand \natexlab [1]{#1}%
\providecommand \enquote  [1]{``#1''}%
\providecommand \bibnamefont  [1]{#1}%
\providecommand \bibfnamefont [1]{#1}%
\providecommand \citenamefont [1]{#1}%
\providecommand \href@noop [0]{\@secondoftwo}%
\providecommand \href [0]{\begingroup \@sanitize@url \@href}%
\providecommand \@href[1]{\@@startlink{#1}\@@href}%
\providecommand \@@href[1]{\endgroup#1\@@endlink}%
\providecommand \@sanitize@url [0]{\catcode `\\12\catcode `\$12\catcode `\&12\catcode `\#12\catcode `\^12\catcode `\_12\catcode `\%12\relax}%
\providecommand \@@startlink[1]{}%
\providecommand \@@endlink[0]{}%
\providecommand \url  [0]{\begingroup\@sanitize@url \@url }%
\providecommand \@url [1]{\endgroup\@href {#1}{\urlprefix }}%
\providecommand \urlprefix  [0]{URL }%
\providecommand \Eprint [0]{\href }%
\providecommand \doibase [0]{http://dx.doi.org/}%
\providecommand \selectlanguage [0]{\@gobble}%
\providecommand \bibinfo  [0]{\@secondoftwo}%
\providecommand \bibfield  [0]{\@secondoftwo}%
\providecommand \translation [1]{[#1]}%
\providecommand \BibitemOpen [0]{}%
\providecommand \bibitemStop [0]{}%
\providecommand \bibitemNoStop [0]{.\EOS\space}%
\providecommand \EOS [0]{\spacefactor3000\relax}%
\providecommand \BibitemShut  [1]{\csname bibitem#1\endcsname}%
\let\auto@bib@innerbib\@empty
\bibitem [{\citenamefont {Bouchaud}\ and\ \citenamefont {Dean}(1995)}]{BD}%
  \BibitemOpen
  \bibfield  {author} {\bibinfo {author} {\bibfnamefont {J.-P.}\ \bibnamefont {Bouchaud}}\ and\ \bibinfo {author} {\bibfnamefont {D.~S.}\ \bibnamefont {Dean}},\ }\bibfield  {title} {\enquote {\bibinfo {title} {Aging on {P}arisi's tree},}\ }\href@noop {} {\bibfield  {journal} {\bibinfo  {journal} {Journal de Physique I}\ }\textbf {\bibinfo {volume} {5}},\ \bibinfo {pages} {265--286} (\bibinfo {year} {1995})}\BibitemShut {NoStop}%
\bibitem [{\citenamefont {M\'ezard}, \citenamefont {Parisi},\ and\ \citenamefont {Virasoro}(1987)}]{MPV}%
  \BibitemOpen
  \bibfield  {author} {\bibinfo {author} {\bibfnamefont {M.}~\bibnamefont {M\'ezard}}, \bibinfo {author} {\bibfnamefont {G.}~\bibnamefont {Parisi}}, \ and\ \bibinfo {author} {\bibfnamefont {M.~A.}\ \bibnamefont {Virasoro}},\ }\href@noop {} {\emph {\bibinfo {title} {Spin glass theory and beyond}}},\ World Scientific Lecture Notes in Physics\ (\bibinfo  {publisher} {World Scientific Publishing Co., Inc., Teaneck, NJ},\ \bibinfo {year} {1987})\ pp.\ \bibinfo {pages} {xiv+461}\BibitemShut {NoStop}%
\bibitem [{\citenamefont {Fontes}\ and\ \citenamefont {Mathieu}(2008)}]{FM08}%
  \BibitemOpen
  \bibfield  {author} {\bibinfo {author} {\bibfnamefont {L.~R.}\ \bibnamefont {Fontes}}\ and\ \bibinfo {author} {\bibfnamefont {P.}~\bibnamefont {Mathieu}},\ }\bibfield  {title} {\enquote {\bibinfo {title} {K-processes, scaling limit and aging for the trap model in the complete graph},}\ }\href@noop {} {\bibfield  {journal} {\bibinfo  {journal} {The Annals of Probability}\ }\textbf {\bibinfo {volume} {36}},\ \bibinfo {pages} {1322--1358} (\bibinfo {year} {2008})}\BibitemShut {NoStop}%
\bibitem [{\citenamefont {Fontes}\ and\ \citenamefont {Lima}(2009{\natexlab{a}})}]{FL1}%
  \BibitemOpen
  \bibfield  {author} {\bibinfo {author} {\bibfnamefont {L.~R.}\ \bibnamefont {Fontes}}\ and\ \bibinfo {author} {\bibfnamefont {P.~H.~S.}\ \bibnamefont {Lima}},\ }\bibfield  {title} {\enquote {\bibinfo {title} {Convergence of symmetric trap models in the hypercube},}\ }in\ \href@noop {} {\emph {\bibinfo {booktitle} {New Trends in Mathematical Physics}}},\ \bibinfo {editor} {edited by\ \bibinfo {editor} {\bibfnamefont {V.}~\bibnamefont {Sidoravi{\v{c}}ius}}}\ (\bibinfo  {publisher} {Springer Netherlands},\ \bibinfo {address} {Dordrecht},\ \bibinfo {year} {2009})\ pp.\ \bibinfo {pages} {285--297}\BibitemShut {NoStop}%
\bibitem [{\citenamefont {Bezerra}\ \emph {et~al.}(2012)\citenamefont {Bezerra}, \citenamefont {Fontes}, \citenamefont {Gava}, \citenamefont {Gayrard},\ and\ \citenamefont {Mathieu}}]{BFGGM}%
  \BibitemOpen
  \bibfield  {author} {\bibinfo {author} {\bibfnamefont {S.~C.}\ \bibnamefont {Bezerra}}, \bibinfo {author} {\bibfnamefont {L.~R.}\ \bibnamefont {Fontes}}, \bibinfo {author} {\bibfnamefont {R.~J.}\ \bibnamefont {Gava}}, \bibinfo {author} {\bibfnamefont {V.}~\bibnamefont {Gayrard}}, \ and\ \bibinfo {author} {\bibfnamefont {P.}~\bibnamefont {Mathieu}},\ }\bibfield  {title} {\enquote {\bibinfo {title} {Scaling limits and aging for asymmetric trap models on the complete graph and {$K$} {P}rocesses},}\ }\href@noop {} {\bibfield  {journal} {\bibinfo  {journal} {ALEA Lat. Am. J. Probab. Math. Stat.}\ }\textbf {\bibinfo {volume} {9}},\ \bibinfo {pages} {303--321} (\bibinfo {year} {2012})}\BibitemShut {NoStop}%
\bibitem [{\citenamefont {Fontes}\ and\ \citenamefont {Peixoto}(2013)}]{FP13}%
  \BibitemOpen
  \bibfield  {author} {\bibinfo {author} {\bibfnamefont {L.~R.}\ \bibnamefont {Fontes}}\ and\ \bibinfo {author} {\bibfnamefont {G.~R.~C.}\ \bibnamefont {Peixoto}},\ }\bibfield  {title} {\enquote {\bibinfo {title} {Elementary results on {K} processes with weights},}\ }\href@noop {} {\bibfield  {journal} {\bibinfo  {journal} {Markov Process. Related Fields}\ }\textbf {\bibinfo {volume} {19}},\ \bibinfo {pages} {343--370} (\bibinfo {year} {2013})}\BibitemShut {NoStop}%
\bibitem [{\citenamefont {Ben~Arous}, \citenamefont {Bovier},\ and\ \citenamefont {{\v{C}}ern{\`y}}(2008)}]{ABC}%
  \BibitemOpen
  \bibfield  {author} {\bibinfo {author} {\bibfnamefont {G.}~\bibnamefont {Ben~Arous}}, \bibinfo {author} {\bibfnamefont {A.}~\bibnamefont {Bovier}}, \ and\ \bibinfo {author} {\bibfnamefont {J.}~\bibnamefont {{\v{C}}ern{\`y}}},\ }\bibfield  {title} {\enquote {\bibinfo {title} {Universality of random energy model-like ageing in mean field spin glasses},}\ }\href@noop {} {\bibfield  {journal} {\bibinfo  {journal} {Journal of Statistical Mechanics: Theory and Experiment}\ }\textbf {\bibinfo {volume} {2008}},\ \bibinfo {pages} {L04003} (\bibinfo {year} {2008})}\BibitemShut {NoStop}%
\bibitem [{\citenamefont {Ben~Arous}, \citenamefont {Bovier},\ and\ \citenamefont {Gayrard}(2003{\natexlab{a}})}]{ABG1}%
  \BibitemOpen
  \bibfield  {author} {\bibinfo {author} {\bibfnamefont {G.}~\bibnamefont {Ben~Arous}}, \bibinfo {author} {\bibfnamefont {A.}~\bibnamefont {Bovier}}, \ and\ \bibinfo {author} {\bibfnamefont {V.}~\bibnamefont {Gayrard}},\ }\bibfield  {title} {\enquote {\bibinfo {title} {Glauber dynamics of the random energy model. {I}. {M}etastable motion on the extreme states},}\ }\href {\doibase 10.1007/s00220-003-0798-4} {\bibfield  {journal} {\bibinfo  {journal} {Comm. Math. Phys.}\ }\textbf {\bibinfo {volume} {235}},\ \bibinfo {pages} {379--425} (\bibinfo {year} {2003}{\natexlab{a}})}\BibitemShut {NoStop}%
\bibitem [{\citenamefont {Ben~Arous}, \citenamefont {Bovier},\ and\ \citenamefont {Gayrard}(2003{\natexlab{b}})}]{ABG2}%
  \BibitemOpen
  \bibfield  {author} {\bibinfo {author} {\bibfnamefont {G.}~\bibnamefont {Ben~Arous}}, \bibinfo {author} {\bibfnamefont {A.}~\bibnamefont {Bovier}}, \ and\ \bibinfo {author} {\bibfnamefont {V.}~\bibnamefont {Gayrard}},\ }\bibfield  {title} {\enquote {\bibinfo {title} {Glauber dynamics of the random energy model. {II}. {A}ging below the critical temperature},}\ }\href {\doibase 10.1007/s00220-003-0799-3} {\bibfield  {journal} {\bibinfo  {journal} {Comm. Math. Phys.}\ }\textbf {\bibinfo {volume} {236}},\ \bibinfo {pages} {1--54} (\bibinfo {year} {2003}{\natexlab{b}})}\BibitemShut {NoStop}%
\bibitem [{\citenamefont {Bovier}\ and\ \citenamefont {Faggionato}(2005)}]{BF}%
  \BibitemOpen
  \bibfield  {author} {\bibinfo {author} {\bibfnamefont {A.}~\bibnamefont {Bovier}}\ and\ \bibinfo {author} {\bibfnamefont {A.}~\bibnamefont {Faggionato}},\ }\bibfield  {title} {\enquote {\bibinfo {title} {Spectral characterization of aging: the {REM}-like trap model},}\ }\href {\doibase 10.1214/105051605000000359} {\bibfield  {journal} {\bibinfo  {journal} {Ann. Appl. Probab.}\ }\textbf {\bibinfo {volume} {15}},\ \bibinfo {pages} {1997--2037} (\bibinfo {year} {2005})}\BibitemShut {NoStop}%
\bibitem [{\citenamefont {Sasaki}\ and\ \citenamefont {Nemoto}(2000)}]{SN}%
  \BibitemOpen
  \bibfield  {author} {\bibinfo {author} {\bibfnamefont {M.}~\bibnamefont {Sasaki}}\ and\ \bibinfo {author} {\bibfnamefont {K.}~\bibnamefont {Nemoto}},\ }\bibfield  {title} {\enquote {\bibinfo {title} {Analysis on aging in the generalized random energy model},}\ }\href@noop {} {\bibfield  {journal} {\bibinfo  {journal} {Journal of the Physical Society of Japan}\ }\textbf {\bibinfo {volume} {69}},\ \bibinfo {pages} {3045--3050} (\bibinfo {year} {2000})}\BibitemShut {NoStop}%
\bibitem [{\citenamefont {Fontes}, \citenamefont {Gava},\ and\ \citenamefont {Gayrard}(2014)}]{FGG}%
  \BibitemOpen
  \bibfield  {author} {\bibinfo {author} {\bibfnamefont {L.~R.}\ \bibnamefont {Fontes}}, \bibinfo {author} {\bibfnamefont {R.~J.}\ \bibnamefont {Gava}}, \ and\ \bibinfo {author} {\bibfnamefont {V.}~\bibnamefont {Gayrard}},\ }\bibfield  {title} {\enquote {\bibinfo {title} {The {K}-process on a tree as a scaling limit of the {GREM}-like trap model},}\ }\href@noop {} {\bibfield  {journal} {\bibinfo  {journal} {Annals of Applied Probability}\ }\textbf {\bibinfo {volume} {24}},\ \bibinfo {pages} {857--897} (\bibinfo {year} {2014})}\BibitemShut {NoStop}%
\bibitem [{\citenamefont {Gayrard}\ and\ \citenamefont {G\"un}(2016)}]{GG15}%
  \BibitemOpen
  \bibfield  {author} {\bibinfo {author} {\bibfnamefont {V.}~\bibnamefont {Gayrard}}\ and\ \bibinfo {author} {\bibfnamefont {O.}~\bibnamefont {G\"un}},\ }\bibfield  {title} {\enquote {\bibinfo {title} {Aging in the {GREM}-like trap model},}\ }\href@noop {} {\bibfield  {journal} {\bibinfo  {journal} {Markov Process. Related Fields}\ }\textbf {\bibinfo {volume} {22}},\ \bibinfo {pages} {165--202} (\bibinfo {year} {2016})}\BibitemShut {NoStop}%
\bibitem [{\citenamefont {Fontes}\ and\ \citenamefont {Gayrard}(2019)}]{FG}%
  \BibitemOpen
  \bibfield  {author} {\bibinfo {author} {\bibfnamefont {L.~R.}\ \bibnamefont {Fontes}}\ and\ \bibinfo {author} {\bibfnamefont {V.}~\bibnamefont {Gayrard}},\ }\bibfield  {title} {\enquote {\bibinfo {title} {Asymptotic behavior and aging of a low temperature cascading 2-{GREM} dynamics at extreme time scales},}\ }\href {\doibase 10.1214/19-ejp395} {\bibfield  {journal} {\bibinfo  {journal} {Electron. J. Probab.}\ }\textbf {\bibinfo {volume} {24}},\ \bibinfo {pages} {Paper No. 142, 50} (\bibinfo {year} {2019})}\BibitemShut {NoStop}%
\bibitem [{\citenamefont {Fontes}, \citenamefont {Fr\'ometa},\ and\ \citenamefont {Zuazn\'abar}(2022)}]{FFZ}%
  \BibitemOpen
  \bibfield  {author} {\bibinfo {author} {\bibfnamefont {L.~R.}\ \bibnamefont {Fontes}}, \bibinfo {author} {\bibfnamefont {S.}~\bibnamefont {Fr\'ometa}}, \ and\ \bibinfo {author} {\bibfnamefont {L.}~\bibnamefont {Zuazn\'abar}},\ }\bibfield  {title} {\enquote {\bibinfo {title} {Asymptotic behavior of a low-temperature non-cascading 2-{GREM} dynamics at extreme time scales},}\ }\href {\doibase 10.3150/21-bej1435} {\bibfield  {journal} {\bibinfo  {journal} {Bernoulli}\ }\textbf {\bibinfo {volume} {28}},\ \bibinfo {pages} {2716--2743} (\bibinfo {year} {2022})}\BibitemShut {NoStop}%
\bibitem [{\citenamefont {Derrida}(1980)}]{D80}%
  \BibitemOpen
  \bibfield  {author} {\bibinfo {author} {\bibfnamefont {B.}~\bibnamefont {Derrida}},\ }\bibfield  {title} {\enquote {\bibinfo {title} {Random-energy model: Limit of a family of disordered models},}\ }\href@noop {} {\bibfield  {journal} {\bibinfo  {journal} {Physical Review Letters}\ }\textbf {\bibinfo {volume} {45}},\ \bibinfo {pages} {79} (\bibinfo {year} {1980})}\BibitemShut {NoStop}%
\bibitem [{\citenamefont {Derrida}(1985)}]{D85}%
  \BibitemOpen
  \bibfield  {author} {\bibinfo {author} {\bibfnamefont {B.}~\bibnamefont {Derrida}},\ }\bibfield  {title} {\enquote {\bibinfo {title} {A generalization of the random energy model which includes correlations between energies},}\ }\href@noop {} {\bibfield  {journal} {\bibinfo  {journal} {Journal de Physique Lettres}\ }\textbf {\bibinfo {volume} {46}},\ \bibinfo {pages} {401--407} (\bibinfo {year} {1985})}\BibitemShut {NoStop}%
\bibitem [{\citenamefont {Fontes}\ and\ \citenamefont {Peixoto}(2021)}]{FP21}%
  \BibitemOpen
  \bibfield  {author} {\bibinfo {author} {\bibfnamefont {L.~R.}\ \bibnamefont {Fontes}}\ and\ \bibinfo {author} {\bibfnamefont {G.~R.~C.}\ \bibnamefont {Peixoto}},\ }\bibfield  {title} {\enquote {\bibinfo {title} {Infinite level {GREM}-like {K} {P}rocesses existence and convergence},}\ }\href {\doibase 10.1007/s10955-021-02713-5} {\bibfield  {journal} {\bibinfo  {journal} {J. Stat. Phys.}\ }\textbf {\bibinfo {volume} {182}},\ \bibinfo {pages} {Paper No. 50, 31} (\bibinfo {year} {2021})}\BibitemShut {NoStop}%
\bibitem [{\citenamefont {Bovier}\ and\ \citenamefont {Gayrard}(2013)}]{BG}%
  \BibitemOpen
  \bibfield  {author} {\bibinfo {author} {\bibfnamefont {A.}~\bibnamefont {Bovier}}\ and\ \bibinfo {author} {\bibfnamefont {V.}~\bibnamefont {Gayrard}},\ }\bibfield  {title} {\enquote {\bibinfo {title} {Convergence of clock processes in random environments and ageing in the {$p$}-spin {SK} model},}\ }\href {\doibase 10.1214/11-AOP705} {\bibfield  {journal} {\bibinfo  {journal} {Ann. Probab.}\ }\textbf {\bibinfo {volume} {41}},\ \bibinfo {pages} {817--847} (\bibinfo {year} {2013})}\BibitemShut {NoStop}%
\bibitem [{\citenamefont {Fontes}, \citenamefont {Isopi},\ and\ \citenamefont {Newman}(2002)}]{FIN}%
  \BibitemOpen
  \bibfield  {author} {\bibinfo {author} {\bibfnamefont {L.~R.}\ \bibnamefont {Fontes}}, \bibinfo {author} {\bibfnamefont {M.}~\bibnamefont {Isopi}}, \ and\ \bibinfo {author} {\bibfnamefont {C.~M.}\ \bibnamefont {Newman}},\ }\bibfield  {title} {\enquote {\bibinfo {title} {Random walks with strongly inhomogeneous rates and singular diffusions: convergence, localization and aging in one dimension},}\ }\href@noop {} {\bibfield  {journal} {\bibinfo  {journal} {The Annals of Probability}\ }\textbf {\bibinfo {volume} {30}},\ \bibinfo {pages} {579--604} (\bibinfo {year} {2002})}\BibitemShut {NoStop}%
\bibitem [{\citenamefont {Ben~Arous}\ and\ \citenamefont {{\v C}ern\'y}(2005)}]{AC05}%
  \BibitemOpen
  \bibfield  {author} {\bibinfo {author} {\bibfnamefont {G.}~\bibnamefont {Ben~Arous}}\ and\ \bibinfo {author} {\bibfnamefont {J.}~\bibnamefont {{\v C}ern\'y}},\ }\bibfield  {title} {\enquote {\bibinfo {title} {Bouchaud's model exhibits two different aging regimes in dimension one},}\ }\href {\doibase 10.1214/105051605000000124} {\bibfield  {journal} {\bibinfo  {journal} {Ann. Appl. Probab.}\ }\textbf {\bibinfo {volume} {15}},\ \bibinfo {pages} {1161--1192} (\bibinfo {year} {2005})}\BibitemShut {NoStop}%
\bibitem [{\citenamefont {Ben~Arous}, \citenamefont {{\v C}ern\'y},\ and\ \citenamefont {Mountford}(2006)}]{ACM}%
  \BibitemOpen
  \bibfield  {author} {\bibinfo {author} {\bibfnamefont {G.}~\bibnamefont {Ben~Arous}}, \bibinfo {author} {\bibfnamefont {J.}~\bibnamefont {{\v C}ern\'y}}, \ and\ \bibinfo {author} {\bibfnamefont {T.}~\bibnamefont {Mountford}},\ }\bibfield  {title} {\enquote {\bibinfo {title} {Aging in two-dimensional {B}ouchaud's model},}\ }\href {\doibase 10.1007/s00440-004-0408-1} {\bibfield  {journal} {\bibinfo  {journal} {Probab. Theory Related Fields}\ }\textbf {\bibinfo {volume} {134}},\ \bibinfo {pages} {1--43} (\bibinfo {year} {2006})}\BibitemShut {NoStop}%
\bibitem [{\citenamefont {Ben~Arous}\ and\ \citenamefont {{\v C}ern\'y}(2007)}]{AC07}%
  \BibitemOpen
  \bibfield  {author} {\bibinfo {author} {\bibfnamefont {G.}~\bibnamefont {Ben~Arous}}\ and\ \bibinfo {author} {\bibfnamefont {J.}~\bibnamefont {{\v C}ern\'y}},\ }\bibfield  {title} {\enquote {\bibinfo {title} {Scaling limit for trap models on {$\Bbb Z^d$}},}\ }\href {\doibase 10.1214/009117907000000024} {\bibfield  {journal} {\bibinfo  {journal} {Ann. Probab.}\ }\textbf {\bibinfo {volume} {35}},\ \bibinfo {pages} {2356--2384} (\bibinfo {year} {2007})}\BibitemShut {NoStop}%
\bibitem [{\citenamefont {Barlow}\ and\ \citenamefont {{\v C}ern\'y}(2011)}]{BC}%
  \BibitemOpen
  \bibfield  {author} {\bibinfo {author} {\bibfnamefont {M.~T.}\ \bibnamefont {Barlow}}\ and\ \bibinfo {author} {\bibfnamefont {J.}~\bibnamefont {{\v C}ern\'y}},\ }\bibfield  {title} {\enquote {\bibinfo {title} {Convergence to fractional kinetics for random walks associated with unbounded conductances},}\ }\href {\doibase 10.1007/s00440-009-0257-z} {\bibfield  {journal} {\bibinfo  {journal} {Probab. Theory Related Fields}\ }\textbf {\bibinfo {volume} {149}},\ \bibinfo {pages} {639--673} (\bibinfo {year} {2011})}\BibitemShut {NoStop}%
\bibitem [{\citenamefont {Fontes}\ and\ \citenamefont {Mathieu}(2014)}]{FM14}%
  \BibitemOpen
  \bibfield  {author} {\bibinfo {author} {\bibfnamefont {L.~R.~G.}\ \bibnamefont {Fontes}}\ and\ \bibinfo {author} {\bibfnamefont {P.}~\bibnamefont {Mathieu}},\ }\bibfield  {title} {\enquote {\bibinfo {title} {On the dynamics of trap models in {$\Bbb Z^d$}},}\ }\href {\doibase 10.1112/plms/pdt064} {\bibfield  {journal} {\bibinfo  {journal} {Proc. Lond. Math. Soc. (3)}\ }\textbf {\bibinfo {volume} {108}},\ \bibinfo {pages} {1562--1592} (\bibinfo {year} {2014})}\BibitemShut {NoStop}%
\bibitem [{Note1()}]{Note1}%
  \BibitemOpen
  \bibinfo {note} {The derivation of~\protect \eqref {eqbd} is fairly straightforward, and will not play a role in our results, so we leave it as an exercise to the interested reader.}\BibitemShut {Stop}%
\bibitem [{\citenamefont {Gava}(2011)}]{Gav}%
  \BibitemOpen
  \bibfield  {author} {\bibinfo {author} {\bibfnamefont {R.~J.}\ \bibnamefont {Gava}},\ }\href {https://doi.org/10.11606/T.45.2011.tde-24092014-224507} {\emph {\bibinfo {title} {Scaling limit of the trap model on a tree}}}\ (\bibinfo  {publisher} {University of São Paulo},\ \bibinfo {year} {2011})\ p.~\bibinfo {pages} {77},\ \bibinfo {note} {thesis (Ph.D.)--University of S\~ao Paulo (in portuguese)}\BibitemShut {NoStop}%
\bibitem [{Note2()}]{Note2}%
  \BibitemOpen
  \bibinfo {note} {{We recall that a subordinator is a non decreasing process whose increments are independent and stationary; an $\alpha $-stable subordinator have increments with an $\alpha $-stable distribution\cite {B99}.}}\BibitemShut {Stop}%
\bibitem [{\citenamefont {Hernández~Delgado}(2024)}]{H}%
  \BibitemOpen
  \bibfield  {author} {\bibinfo {author} {\bibfnamefont {A.~K.}\ \bibnamefont {Hernández~Delgado}},\ }\href {https://doi:10.11606/T.45.2024.tde-09102024-140532} {\emph {\bibinfo {title} {Scaling limits and aging for a representation of the Bouchaud and Dean model on a k-level tree}}}\ (\bibinfo  {publisher} {University of São Paulo},\ \bibinfo {year} {2024})\ p.~\bibinfo {pages} {62},\ \bibinfo {note} {thesis (Ph.D.)--University of S\~ao Paulo}\BibitemShut {NoStop}%
\bibitem [{Note3()}]{Note3}%
  \BibitemOpen
  \bibinfo {note} {{See Remark~\ref {nft} at the end for a brief discussion of other volume scalings.}}\BibitemShut {Stop}%
\bibitem [{Note4()}]{Note4}%
  \BibitemOpen
  \bibinfo {note} {Where, as already mentioned above, the GLTM was analyzed (in the specific ergodic regime, with fine tuning); \protect \eqref {eqsn_inf}, or an extension of it to more levels, is the equilibrium of the limiting rescaled dynamics}\BibitemShut {NoStop}%
\bibitem [{\citenamefont {Ben~Arous}\ and\ \citenamefont {G\"un}(2012)}]{AG}%
  \BibitemOpen
  \bibfield  {author} {\bibinfo {author} {\bibfnamefont {G.}~\bibnamefont {Ben~Arous}}\ and\ \bibinfo {author} {\bibfnamefont {O.}~\bibnamefont {G\"un}},\ }\bibfield  {title} {\enquote {\bibinfo {title} {Universality and extremal aging for dynamics of spin glasses on subexponential time scales},}\ }\href {\doibase 10.1002/cpa.20372} {\bibfield  {journal} {\bibinfo  {journal} {Comm. Pure Appl. Math.}\ }\textbf {\bibinfo {volume} {65}},\ \bibinfo {pages} {77--127} (\bibinfo {year} {2012})}\BibitemShut {NoStop}%
\bibitem [{\citenamefont {Bovier}, \citenamefont {Gayrard},\ and\ \citenamefont {{\v{S}}vejda}(2013)}]{BGS}%
  \BibitemOpen
  \bibfield  {author} {\bibinfo {author} {\bibfnamefont {A.}~\bibnamefont {Bovier}}, \bibinfo {author} {\bibfnamefont {V.}~\bibnamefont {Gayrard}}, \ and\ \bibinfo {author} {\bibfnamefont {A.}~\bibnamefont {{\v{S}}vejda}},\ }\bibfield  {title} {\enquote {\bibinfo {title} {Convergence to extremal processes in random environments and extremal ageing in {SK} models},}\ }\href {\doibase 10.1007/s00440-012-0456-x} {\bibfield  {journal} {\bibinfo  {journal} {Probab. Theory Related Fields}\ }\textbf {\bibinfo {volume} {157}},\ \bibinfo {pages} {251--283} (\bibinfo {year} {2013})}\BibitemShut {NoStop}%
\bibitem [{Note5()}]{Note5}%
  \BibitemOpen
  \bibinfo {note} {{Given a trajectory in such an event, one may approximate it by trajectories in path space which do not are not constant in the given interval, and are thus not in the event.}}\BibitemShut {Stop}%
\bibitem [{\citenamefont {Whitt}(2002)}]{WW}%
  \BibitemOpen
  \bibfield  {author} {\bibinfo {author} {\bibfnamefont {W.}~\bibnamefont {Whitt}},\ }\href@noop {} {\emph {\bibinfo {title} {Stochastic-process limits: an introduction to stochastic-process limits and their application to queues}}}\ (\bibinfo  {publisher} {Springer Science \& Business Media},\ \bibinfo {year} {2002})\BibitemShut {NoStop}%
\bibitem [{Note6()}]{Note6}%
  \BibitemOpen
  \bibinfo {note} {And may be finite, and even empty, for some $l\geq 1$}\BibitemShut {NoStop}%
\bibitem [{Note7()}]{Note7}%
  \BibitemOpen
  \bibinfo {note} {Equality here is not allowed by~\protect \eqref {a}.}\BibitemShut {Stop}%
\bibitem [{\citenamefont {Ethier}\ and\ \citenamefont {Kurtz}(1986)}]{EK}%
  \BibitemOpen
  \bibfield  {author} {\bibinfo {author} {\bibfnamefont {S.~N.}\ \bibnamefont {Ethier}}\ and\ \bibinfo {author} {\bibfnamefont {T.~G.}\ \bibnamefont {Kurtz}},\ }\href {\doibase 10.1002/9780470316658} {\emph {\bibinfo {title} {Markov processes.~Characterization and convergence}}},\ Wiley Series in Probability and Mathematical Statistics: Probability and Mathematical Statistics\ (\bibinfo  {publisher} {John Wiley \& Sons, Inc., New York},\ \bibinfo {year} {1986})\ pp.\ \bibinfo {pages} {x+534}\BibitemShut {NoStop}%
\bibitem [{Note8()}]{Note8}%
  \BibitemOpen
  \bibinfo {note} {Dropping the prime superscript on $n$}\BibitemShut {NoStop}%
\bibitem [{\citenamefont {Jacod}\ and\ \citenamefont {Shiryaev}(2013)}]{J13}%
  \BibitemOpen
  \bibfield  {author} {\bibinfo {author} {\bibfnamefont {J.}~\bibnamefont {Jacod}}\ and\ \bibinfo {author} {\bibfnamefont {A.}~\bibnamefont {Shiryaev}},\ }\href@noop {} {\emph {\bibinfo {title} {Limit theorems for stochastic processes}}},\ Vol.\ \bibinfo {volume} {288}\ (\bibinfo  {publisher} {Springer Science \& Business Media},\ \bibinfo {year} {2013})\BibitemShut {NoStop}%
\bibitem [{\citenamefont {Hu}, \citenamefont {Szynal},\ and\ \citenamefont {Volodin}(1998)}]{HSV}%
  \BibitemOpen
  \bibfield  {author} {\bibinfo {author} {\bibfnamefont {T.-C.}\ \bibnamefont {Hu}}, \bibinfo {author} {\bibfnamefont {D.}~\bibnamefont {Szynal}}, \ and\ \bibinfo {author} {\bibfnamefont {A.}~\bibnamefont {Volodin}},\ }\bibfield  {title} {\enquote {\bibinfo {title} {A note on complete convergence for arrays},}\ }\href@noop {} {\bibfield  {journal} {\bibinfo  {journal} {Statistics \& Probability Letters}\ }\textbf {\bibinfo {volume} {38}},\ \bibinfo {pages} {27--31} (\bibinfo {year} {1998})}\BibitemShut {NoStop}%
\bibitem [{Note9()}]{Note9}%
  \BibitemOpen
  \bibinfo {note} {It does not depend on the first portion of the trajectory, but this is of course not a problem}\BibitemShut {NoStop}%
\bibitem [{\citenamefont {Erickson}(1970)}]{E}%
  \BibitemOpen
  \bibfield  {author} {\bibinfo {author} {\bibfnamefont {K.~B.}\ \bibnamefont {Erickson}},\ }\bibfield  {title} {\enquote {\bibinfo {title} {Strong renewal theorems with infinite mean},}\ }\href {\doibase 10.2307/1995628} {\bibfield  {journal} {\bibinfo  {journal} {Trans. Amer. Math. Soc.}\ }\textbf {\bibinfo {volume} {151}},\ \bibinfo {pages} {263--291} (\bibinfo {year} {1970})}\BibitemShut {NoStop}%
\bibitem [{\citenamefont {Ben~Arous}\ and\ \citenamefont {{\v C}ern\'y}(2006)}]{AC06}%
  \BibitemOpen
  \bibfield  {author} {\bibinfo {author} {\bibfnamefont {G.}~\bibnamefont {Ben~Arous}}\ and\ \bibinfo {author} {\bibfnamefont {J.}~\bibnamefont {{\v C}ern\'y}},\ }\bibfield  {title} {\enquote {\bibinfo {title} {Dynamics of trap models},}\ }in\ \href {\doibase 10.1016/S0924-8099(06)80045-4} {\emph {\bibinfo {booktitle} {Mathematical statistical physics}}}\ (\bibinfo  {publisher} {Elsevier B. V., Amsterdam},\ \bibinfo {year} {2006})\ pp.\ \bibinfo {pages} {331--394}\BibitemShut {NoStop}%
\bibitem [{\citenamefont {Bertoin}(1999)}]{B99}%
  \BibitemOpen
  \bibfield  {author} {\bibinfo {author} {\bibfnamefont {J.}~\bibnamefont {Bertoin}},\ }\bibfield  {title} {\enquote {\bibinfo {title} {Subordinators: examples and applications},}\ }in\ \href {\doibase 10.1007/978-3-540-48115-7\_1} {\emph {\bibinfo {booktitle} {Lectures on probability theory and statistics ({S}aint-{F}lour, 1997)}}},\ \bibinfo {series} {Lecture Notes in Math.}, Vol.\ \bibinfo {volume} {1717}\ (\bibinfo  {publisher} {Springer, Berlin},\ \bibinfo {year} {1999})\ pp.\ \bibinfo {pages} {1--91}\BibitemShut {NoStop}%
\bibitem [{\citenamefont {Bouchaud}(1992)}]{B92}%
  \BibitemOpen
  \bibfield  {author} {\bibinfo {author} {\bibfnamefont {J.-P.}\ \bibnamefont {Bouchaud}},\ }\bibfield  {title} {\enquote {\bibinfo {title} {Weak ergodicity breaking and aging in disordered systems},}\ }\href@noop {} {\bibfield  {journal} {\bibinfo  {journal} {Journal de Physique I}\ }\textbf {\bibinfo {volume} {2}},\ \bibinfo {pages} {1705--1713} (\bibinfo {year} {1992})}\BibitemShut {NoStop}%
\bibitem [{\citenamefont {Durrett}\ and\ \citenamefont {Resnick}(1978)}]{D78}%
  \BibitemOpen
  \bibfield  {author} {\bibinfo {author} {\bibfnamefont {R.}~\bibnamefont {Durrett}}\ and\ \bibinfo {author} {\bibfnamefont {S.~I.}\ \bibnamefont {Resnick}},\ }\bibfield  {title} {\enquote {\bibinfo {title} {Functional limit theorems for dependent variables},}\ }\href@noop {} {\bibfield  {journal} {\bibinfo  {journal} {The Annals of Probability}\ ,\ \bibinfo {pages} {829--846}} (\bibinfo {year} {1978})}\BibitemShut {NoStop}%
\bibitem [{\citenamefont {Durrett}(2019)}]{D19}%
  \BibitemOpen
  \bibfield  {author} {\bibinfo {author} {\bibfnamefont {R.}~\bibnamefont {Durrett}},\ }\href@noop {} {\emph {\bibinfo {title} {Probability: theory and examples}}},\ Vol.~\bibinfo {volume} {49}\ (\bibinfo  {publisher} {Cambridge university press},\ \bibinfo {year} {2019})\BibitemShut {NoStop}%
\bibitem [{\citenamefont {Fontes}\ and\ \citenamefont {Lima}(2009{\natexlab{b}})}]{FL2}%
  \BibitemOpen
  \bibfield  {author} {\bibinfo {author} {\bibfnamefont {L.~R.}\ \bibnamefont {Fontes}}\ and\ \bibinfo {author} {\bibfnamefont {P.~H.~S.}\ \bibnamefont {Lima}},\ }\href {https://arxiv.org/abs/0809.3463} {\enquote {\bibinfo {title} {Convergence of symmetric trap models in the hypercube},}\ } (\bibinfo {year} {2009}{\natexlab{b}}),\ \Eprint {http://arxiv.org/abs/0809.3463} {arXiv:0809.3463 [math.PR]} \BibitemShut {NoStop}%
\bibitem [{\citenamefont {Gayrard}(2010)}]{G10}%
  \BibitemOpen
  \bibfield  {author} {\bibinfo {author} {\bibfnamefont {V.}~\bibnamefont {Gayrard}},\ }\bibfield  {title} {\enquote {\bibinfo {title} {Aging in reversible dynamics of disordered systems. i. emergence of the arcsine law in bouchaud's asymmetric trap model on the complete graph},}\ }\href@noop {} {\bibfield  {journal} {\bibinfo  {journal} {arXiv preprint arXiv:1008.3855}\ } (\bibinfo {year} {2010})}\BibitemShut {NoStop}%
\bibitem [{\citenamefont {Gayrard}(2012)}]{G12}%
  \BibitemOpen
  \bibfield  {author} {\bibinfo {author} {\bibfnamefont {V.}~\bibnamefont {Gayrard}},\ }\bibfield  {title} {\enquote {\bibinfo {title} {Convergence of clock process in random environments and aging in {B}ouchaud's asymmetric trap model on the complete graph},}\ }\href {\doibase 10.1214/EJP.v17-2211} {\bibfield  {journal} {\bibinfo  {journal} {Electron. J. Probab.}\ }\textbf {\bibinfo {volume} {17}},\ \bibinfo {pages} {no. 58, 33} (\bibinfo {year} {2012})}\BibitemShut {NoStop}%
\bibitem [{\citenamefont {Kingman}(1992)}]{K92}%
  \BibitemOpen
  \bibfield  {author} {\bibinfo {author} {\bibfnamefont {J.~F.~C.}\ \bibnamefont {Kingman}},\ }\href@noop {} {\emph {\bibinfo {title} {Poisson processes}}},\ Vol.~\bibinfo {volume} {3}\ (\bibinfo  {publisher} {Clarendon Press},\ \bibinfo {year} {1992})\BibitemShut {NoStop}%
\bibitem [{\citenamefont {Seneta}(2006)}]{S06}%
  \BibitemOpen
  \bibfield  {author} {\bibinfo {author} {\bibfnamefont {E.}~\bibnamefont {Seneta}},\ }\href@noop {} {\emph {\bibinfo {title} {Regularly varying functions}}},\ Vol.\ \bibinfo {volume} {508}\ (\bibinfo  {publisher} {Springer},\ \bibinfo {year} {2006})\BibitemShut {NoStop}%
\end{thebibliography}%

\end{document}